\documentclass[10pt]{amsart}

\usepackage{enumerate, amsmath, amsfonts, amssymb, amsthm, mathtools, thmtools, wasysym, graphics, graphicx, xcolor, frcursive,xparse,comment,bbm}
\usepackage[all]{xy}
\usepackage[colorlinks=true, pdfstartview=FitV, linkcolor=blue, citecolor=blue, urlcolor=blue]{hyperref}

\usepackage{url, hypcap}
\hypersetup{colorlinks=true, citecolor=darkblue, linkcolor=darkblue}
\usepackage{mathrsfs}

\usepackage{tikz}
\usetikzlibrary{arrows,automata}
\usetikzlibrary{patterns,snakes}
\tikzset{
	edge/.style={->,> = latex'}
}
\usetikzlibrary{calc,through,backgrounds,shapes,matrix}
\usepackage{graphicx}

\usepackage[letterpaper,margin=1.25in]{geometry}

\definecolor{darkblue}{rgb}{0.0,0,0.7} 
\definecolor{lightblue}{rgb}{0,135,147}

\definecolor{red}{rgb}{1,0,0}

\definecolor{darkred}{rgb}{0.7,0,0} 
\definecolor{lightgrey}{rgb}{0.7,0.7,0.7} 

\newcommand{\Ideal}{\mathcal{I}}

\newtheorem{theorem}{Theorem}[section]
\newtheorem{proposition}[theorem]{Proposition}
\newtheorem{corollary}[theorem]{Corollary}
\newtheorem{lemma}[theorem]{Lemma}

\theoremstyle{definition}
\newtheorem{definition}[theorem]{Definition}
\newtheorem{example}[theorem]{Example}
\newtheorem{conjecture}[theorem]{Conjecture}

\newtheorem{remark}[theorem]{Remark}

\numberwithin{equation}{section}

\definecolor{darkred}{rgb}{0.7,0,0} 
\newcommand{\defn}[1]{{\color{darkred}\emph{#1}}} 

\usepackage[colorinlistoftodos]{todonotes}

\title[Unified theory for finite Markov chains]
  {Unified theory for finite Markov chains}
  
\author[J.~Rhodes]{John Rhodes}
\address[J. Rhodes]{Department of Mathematics, University of California, Berkeley, CA 94720, U.S.A.}
\email{rhodes@math.berkeley.edu, blvdbastille@gmail.com}

\author[A.~Schilling]{Anne Schilling}
\address[A. Schilling]{Department of Mathematics, UC Davis, One Shields Ave., Davis, CA 95616-8633, U.S.A.}
\email{anne@math.ucdavis.edu}

\date{\today}

\begin{document}

\begin{abstract}
We provide a unified framework to compute the stationary distribution of any finite irreducible Markov chain or 
equivalently of any irreducible random walk on a finite semigroup $S$.
Our methods use geometric finite semigroup theory via the Karnofsky--Rhodes and the McCammond expansions 
of finite semigroups with specified generators; this does not involve any linear algebra.
The original Tsetlin library is obtained by applying the expansions to $P(n)$, the set 
of all subsets of an $n$ element set. Our set-up generalizes previous groundbreaking work involving left-regular bands (or 
$\mathscr{R}$-trivial bands) by Brown and Diaconis, extensions to $\mathscr{R}$-trivial semigroups by Ayyer, Steinberg, 
Thi\'ery and the second author, and important recent work by Chung and Graham. The Karnofsky--Rhodes expansion of 
the right Cayley graph of $S$ in terms of generators yields again a right Cayley graph. The McCammond expansion 
provides normal forms for elements in the expanded $S$. Using our previous results with Silva based on work by Berstel, 
Perrin, Reutenauer, we construct (infinite) semaphore codes on which we can define Markov chains. These semaphore 
codes can be lumped using geometric semigroup theory. Using normal forms and associated Kleene expressions, they 
yield formulas for the stationary distribution of the finite Markov chain of the expanded $S$ and the original $S$. 
Analyzing the normal forms also provides an estimate on the mixing time.
\end{abstract}

\maketitle

\section{Introduction}

The Tsetlin library~\cite{Tsetlin.1963} is a Markov chain, whose states are all permutations $S_n$ of $n$ books on a shelf.
Given an arrangement of books $\sigma \in S_n$, construct $\sigma' \in S_n$ from $\sigma$ by removing book $a$ from 
the shelf and inserting it to the front. To each such transition $\sigma \stackrel{a}{\longrightarrow} \sigma'$, we associate
a probability $x_a$. If the probability $x_a$ is large, it means that book $a$ is popular, whereas if $x_a$ is small, then
book $a$ is unpopular. Running this Markov chain for a while has the effect of accumulating the popular books in
the front. The stationary distribution is the limiting distributing of the books, when one lets the Markov chain run
for a long time. The precise formula was derived by Hendricks~\cite{Hendricks.1972, Hendricks.1973}.

In the meantime, many generalizations of the Tsetlin library have been studied, such as
walks on hyperplane arrangements~\cite{Bidigare.1997,BHR.1999}, Brown's significant generalization to left regular 
bands~\cite{Brown.2000} based on important work by Brown and Diaconis~\cite{BD.1998}, hierarchies of 
libraries~\cite{Bjorner.2008,Bjorner.2009}, edge flipping in graphs~\cite{ChungGraham.2012}, random walks on linear 
extensions of a poset~\cite{AKS.2014}, random walks on general $\mathscr{R}$-trivial semigroups~\cite{ASST.2015}, and
others~\cite{AS.2010, Ayyer.2011,AS.2013,ASST.2015a,PS.2017}.
The main technique, that made the analysis of all of these random walks possible, is the concept of reduced
words of the elements in the underlying semigroup. As pointed out in~\cite[Section 4.4]{ASST.2015} 
and~\cite[Remark 3.1]{MSS.2015} in the context of $\mathscr{R}$-trivial semigroups, this is the Karnofsky--Rhodes 
expansion of the support semilattice introduced in the seminal paper by Brown~\cite{Brown.2000}. This is an example, 
where concepts from semigroup theory were rediscovered in the setting of probability.

As is often the case in mathematics, once there is a toehold, an avalanche of results can follow by applying the results 
of the new field. The theory that we develop in this paper makes it possible
to compute the stationary distribution for any irreducible finite Markov chain. It uses the power of geometric finite semigroup 
theory~\cite{MRS.2011} via the Karnofsky--Rhodes and the McCammond expansions of finite
semigroups and does not use any linear algebra. From the theory of regular languages (that is, finite semigroup 
definable languages), we can define Markov chains on the expanded semigroup using semaphore 
codes~\cite{BPR.2010,RSS.2016}. The Karnofsky--Rhodes and McCammond expansions ensure the existence of 
normal forms for elements (or paths) in the semaphore code. Using Kleene expressions, Zimin words and elementary 
combinatorics, we are able to derive the stationary distribution of all irreducible random walks associated to a finite semigroup.
Since there are only a finite number of states and a finite number of maps between the states, given a probability
distribution there is a decidable algorithm giving the stationary distribution of the Markov chain, which is new.

This generalizes known stationary distributions of random walks and provides an abundance of
new interesting examples of random walks and their stationary distributions. We obtain new examples
by the $\mathsf{bar}$ construction from finite semigroup theory~\cite{LRS.2017} and the technology of the solution
of the Burnside problem by McCammond and others~\cite{McCammond.1991}. In addition, we provide a standard 
interpretation of these constructions to understand how they apply to the real world.

The random walks that we deal with are in general not diagonalizable, unlike in the case of left regular 
bands~\cite{Brown.2000}. Our approach differs in that we start with an irreducible, infinitely countable
random walk, namely the random walk on semaphore codes with the Bernoulli distribution as its stationary
distribution. Using advanced finite semigroup theory, we find projections of these walks via lumping first in the case
when the minimal ideal is left zero. The lumping is allowed thanks to the fact that the Karnofksy--Rhodes expansion 
of a right Cayley graph is itself a right Cayley graph of a finite semigroup. We obtain the general case from the case 
when the minimal ideal is left zero as a limiting case by applying the flat operator~\cite{LRS.2017} and then limiting the 
probability of the new introduced generator to zero. The resulting random walks on finite semigroups are in general not 
diagonalizable, but we can nonetheless compute the stationary distributions. Using the hitting time of semaphore 
codes~\cite{RSS.2016} and Kleene expressions, we can also estimate the mixing time of these walks
via the techniques in~\cite[Lemma 3.6]{ASST.2015}.

\medskip

The paper is organized as follows. Section~\ref{section.walks} is devoted to discrete Markov chains and
their analysis using semigroup theory. We begin with a review of Cayley graphs of finite semigroups 
(Section~\ref{section.right cayley}), Markov chains in the language of semigroups (Section~\ref{section.markov}), and
 random walks on semaphore codes (Section~\ref{section.semaphore}). Ideals in semigroups are intimately related to 
 (possibly infinite) semaphore codes.
We proceed to explain the Karnofksy--Rhodes (Section~\ref{section.KR}) and McCammond expansion 
(Section~\ref{section.mccammond}) of the right Cayley graph (Section~\ref{section.right cayley}) of a finite semigroup $S$ 
with generators $A$. The McCammond expansion guarantees normal forms of all elements in terms of the generators
(Section~\ref{section.normal}). It is possible to lump the random walk on semaphore codes by reducing to 
simple paths without loops (Section~\ref{section.lumping}) in the case when the minimal ideal in the semigroup is
left zero. Using Kleene expressions and Zimin words, we provide explicit expressions for the stationary 
distributions (Section~\ref{section.kleene}). Adding a zero to the semigroup, it is possible to obtain the stationary
distribution for any finite semigroup from the case when the minimal ideal is left zero as a limiting case
(Section~\ref{section.bar and flat}). In Section~\ref{section.mixing time} we provide bounds on the mixing time.
In Section~\ref{section.examples}, we discuss many examples of semigroups and 
how our methods yield the stationary distributions of known and new Markov chains, such as the original Tsetlin library 
(Section~\ref{section.Tsetlin library}), edge flipping on a line (Section~\ref{section.edge flipping}),
cyclic walks using the Rees matrix semigroup (Sections~\ref{section.Bn} 
and~\ref{section.BnZp}), and random walks on general $\mathscr{R}$-trivial semigroups (Section~\ref{section.R trivial}). 
The $\mathsf{bar}$ and flat $\flat$ operations, introduced in Section~\ref{section.bar and flat}, are then used in 
Sections~\ref{section.bar} and~\ref{section.flat} to produce many infinite families of examples of Markov chains.
We conclude in Section~\ref{section.burnside} with examples in the Burnside class.

\subsection*{Acknowledgments}
We are grateful to Arvind Ayyer, Persi Diaconis, Stuart Margolis, Christoph Reutenauer, Ben Steinberg, and Nicolas Thi\'ery 
for helpful discussions and comments. In particular, we thank Persi Diaconis for suggesting the example in
Section~\ref{section.edge flipping} and other improvements to this paper.
The first author thanks the Simons Foundation Collaboration Grants for Mathematicians for travel grant \#313548.
The second author was partially supported by NSF grant  DMS--1500050.

\section{Random walks on semigroups}
\label{section.walks}

Let $S$ be a finite semigroup. We are interested in considering $S$ together with a choice of generators $A$,
denoted by $(S,A)$. Every finite semigroup has a finite set of generators (for example, the elements of $S$ itself, but 
possibly fewer). We will construct Markov chains for $(S,A)$ using this set-up by associating a probability $x_a$ to each 
generator $a\in A$.

\subsection{Cayley graphs}
\label{section.right cayley}

Given a \defn{finite semigroup} $S$ and a set of \defn{generators} $A$, we can view $A$ as a finite, non-empty alphabet. 
Denote by $A^+$ the set of all words $a_1 \ldots a_\ell$ of length $\ell \geqslant 1$ over $A$ with multiplication given
by concatenation. Thus $(A^+,A)$ is the free semigroup with generators $A$. Furthermore, let $A^\star = A^+ \cup \{1\}$, 
so that $A^\star$ is $A^+$ with the identity added; it is the free monoid generated by $A$. The elements of $A^\star$ are 
typically called words. A subset of $A^\star$ is called a \defn{language}.

A semigroup $S$ with multiplication $\cdot$ generated by a subset $A\subseteq S$ determines a surmorphism
\begin{equation}
\label{equation.phi}
	\varphi \colon (A^+,A) \to (S,A)
\end{equation}
mapping $a_1\ldots a_\ell \in A^+$ to $a_1 \cdot a_2 \cdot \ldots \cdot a_\ell \in S$. Given a word $w\in A^+$, we
denote $[w]_S := \varphi(w)$ to avoid the reference to $\varphi$. The pair $(S,A)$ is also sometimes called
an $A$-semigroup (see~\cite[Definition 2.15]{MRS.2011}).

\begin{definition}[Graph]
A \defn{labeled directed graph} $\Gamma$ (or \defn{graph} for short) consists of a vertex set $V(\Gamma)$, an edge 
set $E(\Gamma)$, and a labelling set $A$. An edge $e\in E(\Gamma)$ is a tuple $e=(v,a,w) \in V(\Gamma)\times A \times
V(\Gamma)$. We often also write $e\colon v \stackrel{a}{\longrightarrow} w$.
\end{definition} 

A \defn{path} $p$ from vertex $v$ to vertex $w$ in a graph $\Gamma$ is a sequence of edges
\[
	p = \left(v = v_0 \stackrel{a_1}{\longrightarrow} v_1 \stackrel{a_2}{\longrightarrow} \cdots 
	\stackrel{a_\ell}{\longrightarrow} v_\ell =w\right),
\]
where each tuple $(v_i,a_{i+1},v_{i+1}) \in E(\Gamma)$ for $0\leqslant i<\ell$. The initial (resp. terminal) vertex $v$ 
(resp. $w$) of $p$ is denoted by $\iota(p)$ (resp. $\tau(p)$). The \defn{length} of $p$ is $\ell(p):= \ell$
and $a_1 \ldots a_\ell$ is called the label of the path. If $p$ and $q$ are paths, $\ell(q)=k \leqslant \ell(p)$, and the first 
$k+1$ vertices and $k$ edges of $p$ and $q$ agree, we say that $q$ is an \defn{initial segment} of $p$, written 
$q \subseteq p$.

We can define a preorder $\prec$ on $V(\Gamma)$ by $v \prec w$ if there is a path from $v$ to $w$ in $\Gamma$.
This induces an equivalence relation $\sim$ on $V(\Gamma)$, where $v \sim w$ if $v \prec w$ and $w \prec v$.
A \defn{strongly connected component} of $\Gamma$ is a $\sim$-equivalence class.

\begin{definition}[Rooted graph]
A \defn{rooted graph} is a pair $(\Gamma,r)$, where $\Gamma$ is a graph and $r \in V(\Gamma)$, such that
$r \prec v$ for all $v \in V(\Gamma)$.
\end{definition}

A path is called \defn{simple} if it visits no vertex twice. Empty (or trivial) paths are considered simple.
For a rooted graph $(\Gamma,r)$, let $\mathsf{Simple}(\Gamma,r)$ be the set of simple paths of $\Gamma$
starting at $r$ (including the empty path).

We are now ready to apply this set-up to semigroups. If $S$ is a semigroup, then $S^{\mathbbm{1}}$ denotes $S$ with 
an adjoined identity $\mathbbm{1}$ even if $S$ already has an identity.

\begin{definition}[Right and left Cayley graph]
Let $(S,A)$ be a finite semigroup $S$ together with a set of generators $A$. 
The \defn{right Cayley graph} $\mathsf{RCay}(S,A)$ of $S$ with respect to $A$ is the rooted graph with
vertex set $V(\mathsf{RCay}(S,A)) = S^{\mathbbm{1}}$, root $r=\mathbbm{1} \in S^{\mathbbm{1}}$, and edges 
$s \stackrel{a}{\longrightarrow} s'$ for all $(s,a,s') \in S^{\mathbbm{1}} \times A \times S^{\mathbbm{1}}$, where $s'=sa$ 
in $S^{\mathbbm{1}}$.  The \defn{left Cayley graph} $\mathsf{LCay}(S,A)$ is defined in the analogous fashion
with the only difference that $s'=as$.
\end{definition}

\begin{remark}
Since $(s,a) \in S^{\mathbbm{1}} \times A$ uniquely determines the edge $s  \stackrel{a}{\longrightarrow} sa$ in 
$\mathsf{RCay}(S,A)$, we sometimes also index the edge set as $E(\mathsf{RCay}(S,A)) = S^{\mathbbm{1}} \times A$ for 
right Cayley graphs and analogously for left Cayley graphs.
\end{remark}

An example of a right Cayley graph is given in Figure~\ref{figure.right cayley}.

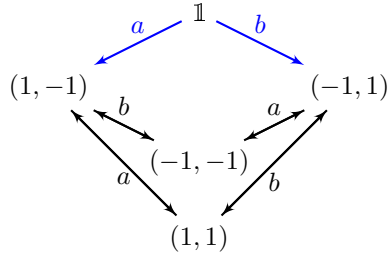
\begin{figure}[t]
\begin{center}
\begin{tikzpicture}[auto]
\node (A) at (0, 0) {$\mathbbm{1}$};
\node (B) at (-2,-1) {$(1,-1)$};
\node(C) at (2,-1) {$(-1,1)$};
\node(D) at (0,-2) {$(-1,-1)$};
\node(E) at (0,-3) {$(1,1)$};
\draw[edge,blue,thick] (A) -- (B) node[midway, above] {$a$};
\draw[edge,blue,thick] (A) -- (C) node[midway, above] {$b$};
\draw[edge,thick] (B) -- (D) node[midway, above] {$b$};
\draw[edge,thick] (D) -- (B);
\draw[edge,thick] (C) -- (D) node[midway, above] {$a$};
\draw[edge,thick] (D) -- (C);
\draw[edge,thick] (B) -- (E) node[midway, below] {$a$};
\draw[edge,thick] (E) -- (B);
\draw[edge,thick] (C) -- (E) node[midway, below] {$b$};
\draw[edge,thick] (E) -- (C);
\end{tikzpicture}
\end{center}
\caption{\label{figure.right cayley}The right Cayley graph $\mathsf{RCay}(Z_2 \times Z_2,\{a,b\})$ of the Klein $4$-group
and generators $a=(1,-1)$ and $b=(-1,1)$. Transition edges are indicated in blue. Double edges mean that right multiplication
by the label for either vertex yields the other vertex.}
\end{figure}

For a semigroup $S$, two elements $s,s'\in S$ are in the same $\mathscr{R}$-class if the corresponding right ideals 
are equal, that is, $s S^{\mathbbm{1}} = s'S^{\mathbbm{1}}$. The strongly connected components of $\mathsf{RCay}(S,A)$ 
are precisely the $\mathscr{R}$-classes of $S^{\mathbbm{1}}$. In other words, the vertices of a strongly connected 
component are exactly the vertices that represent the elements in an $\mathscr{R}$-class of $S^{\mathbbm{1}}$. 
Edges that go between distinct strongly connected components will turn out to play an important role in the 
Karnofksy--Rhodes expansion that we will need later.

\begin{definition}[Transition edges]
Let $\Gamma$ be a graph. Then $e=(v,a,w) \in E(\Gamma)$ with $v,w\in V(\Gamma)$ and $a\in A$ is 
a \defn{transition edge} if $v \not \sim w$. In other words, there is no path from $w$ to $v$ in $\Gamma$.
\end{definition}

In Figure~\ref{figure.right cayley}, the transition edges are indicated in blue. Note that the edges leaving $\mathbbm{1}$
in the right Cayley graph are always transitional. Other edges might or might not be transitional.

\subsection{Markov chains}
\label{section.markov}

A \defn{Markov chain} $\mathcal{M}$ consists of a finite or countable state space $\Omega$ together with transition
probabilities $\mathcal{T}_{s',s}$ for the transition $s \longrightarrow s'$ for $s,s' \in \Omega$.
The matrix $\mathcal{T} = (\mathcal{T}_{s',s})_{s,s'\in \Omega}$ is called the \defn{transition matrix}.
In our convention, the column sums of $\mathcal{T}$ are equal to one, or equivalently, that $\mathcal{T}$ is a 
column-stochastic matrix.

A Markov chain is called \defn{irreducible} if for any $s,s' \in \Omega$ there exists an integer $m$ (possibly depending
on $s$, $s'$) such that $\mathcal{T}_{s',s}^m >0$. In other words, one can get from any state $s$ to any other state $s'$ 
using only steps with positive probability. A state $s\in \Omega$ is called \defn{recurrent} if the system returns
to $s$ in finitely many steps with probability one.

The \defn{stationary distribution} of $\mathcal{M}$ is a vector $\Psi = (\Psi_s)_{s\in \Omega}$ such that 
$\mathcal{T} \Psi = \Psi$ and $\sum_{s \in \Omega} \Psi_s=1$. In other words, $\Psi$ is a right-eigenvector of 
$\mathcal{T}$ with eigenvalue one. If the Markov chain is irreducible, the stationary distribution is unique~\cite{LPW.2009}.

Let us now partition the state space into $(\Omega_1,\ldots,\Omega_\ell)$ such that $\Omega_i \cap
\Omega_j = \emptyset$ for $i\neq j$ and $\bigcup_{i=1}^\ell \Omega_i = \Omega$. One may view
such a partition also as an equivalence relation $s\sim s'$ if $s,s'\in \Omega_i$ for some $1\leqslant i \leqslant \ell$.
We say that $\mathcal{M}$ can be \defn{lumped} with respect to the partition $(\Omega_1,\ldots,
\Omega_\ell)$ if the transition matrix $\mathcal{T}$ satisfies~\cite[Lemma~2.5]{LPW.2009} \cite{KS.1976}
for all $1\leqslant i,j \leqslant \ell$
\begin{equation}
\label{equation.lumping}
	\sum_{t\in \Omega_j} \mathcal{T}_{t,s} = \sum_{t\in \Omega_j} \mathcal{T}_{t,s'}
	\qquad \text{for all $s,s' \in \Omega_i.$}
\end{equation}
The lumped Markov chain is a random walk on the equivalence classes, whose stationary distribution labeled by
$w$ is $\sum_{s \sim w} \Psi_s$.

As explained in~\cite[Proposition 1.5]{LPW.2009} and~\cite[Theorem 2.3]{ASST.2015}, every finite state Markov chain 
$\mathcal{M}$ has a random letter representation, that is, a representation of a semigroup $S$ acting on the left on the 
state space $\Omega$. In this setting, we transition $s \stackrel{a}{\longrightarrow} s'$ with probability 
$0\leqslant x_a\leqslant 1$, where $s, s'\in \Omega$, $a\in S$ and $s'=a.s$ is the action of $a$ on the state $s$. 
Let $A=\{a\in S \mid x_a>0\}$. We assume that $A$ generates $S$; if not, it suffices to consider the subsemigroup
generated by $A$. Note that $\sum_{a\in A} x_a =1$. The \defn{transition matrix} $\mathcal{T}$ of $\mathcal{M}$ is the 
$|\Omega| \times |\Omega|$-matrix 
\begin{equation}
\label{equation.transition matrix}
	\mathcal{T}_{s',s} = \sum_{\substack{a \in A\\ s \stackrel{a}{\longrightarrow} s'}} x_a
	\qquad \text{for $s,s' \in \Omega$.}
\end{equation}
Note that we may assume that the action of $S$ on $\Omega$ is faithful as this does not affect the random
walk.

\begin{definition}[Ideal]
Let $S$ be a semigroup. A two-sided \defn{ideal} $I$ (or ideal for short) is a subset $I \subseteq S$ such that 
$u I v \subseteq I$ for all $u,v \in S^{\mathbbm{1}}$. Similarly, a \defn{left ideal} $I$ is a subset 
$I \subseteq S^{\mathbbm{1}}$ such that $u I \subseteq I$ for all $u\in S^{\mathbbm{1}}$. 
\end{definition}

If $I,J$ are ideals of $S$, then $IJ \subseteq I \cap J$, so that $I \cap J \neq \emptyset$. Hence every
finite semigroup has a unique minimal ideal denoted $K(S)$.
As shown in~\cite{Clifford.Preston.1961,KRT.1968}, the minimal ideal $K(S)$ of a finite semigroup $S$ is the disjoint 
union of all the minimal left ideals of $S$ and the Rees Theorem applies. By~\cite[Remark 2.8]{ASST.2015} 
the faithful left action of $S$ generated by $A$ on $\Omega$ is isomorphic to the left action of $S$ on $K(S)$.

For a finite $A$-semigroup $(S,A)$, let $\mathcal{M}(S,A)$ be the Markov chain, where the transition
$s \stackrel{a}{\longrightarrow} s'$ for $s,s'\in S$ and $a\in A$ is given by $s'=as$ in the left Cayley graph
with probability $0<x_a\leqslant 1$. Note that we are assuming that all probabilities $x_a$ for $a\in A$
are nonzero. Then it was shown in~\cite{HM.2011} (see also~\cite[Proposition 3.2]{ASST.2015}) that the
recurrent states of $\mathcal{M}(S,A)$ are the elements in $K(S)$. Furthermore, the connected components
of the recurrent states in the random walk are the minimal left ideals of $S$. The restriction of the random walk 
to any minimal left ideal is irreducible. Moreover, the chain so obtained is independent of the chosen minimal left ideal.
This random walk and the random walk with states a left ideal $L$ of $K(S)$ and $S$ acting on the left made faithful, 
that is $x \stackrel{a}{\longrightarrow} y$ for $x \in L$ and $y = ax$, are essentially the same. So we may not distinguish 
the two cases. 

In the following, we first treat the case when $K(S)$ is left zero (that is, $xy=x$ for all $x,y\in K(S)$) using semaphore 
codes, the Karnofsky--Rhodes and McCammond expansion of the right Cayley graph of $(S,A)$, and Kleene expressions.
In Corollary~\ref{corollary.stationary general}, we add a zero to the semigroup and generators to deduce
the case for general $K(S)$ from the case when $K(S)$ is left zero.

\subsection{Semaphore codes}
\label{section.semaphore}

Ideals in a semigroup are related to semaphore codes~\cite{BPR.2010,RSS.2016}. They also give rise to Markov 
chains since they allow for a left action, as we will now explain. As before, let $A$ be a finite, non-empty alphabet. The 
semigroup $A^+$ has three orders: ``is a suffix", ``is a prefix", and ``is a factor". In particular, for $u,v \in A^+$
\begin{equation*}
\begin{split}
	& \text{$u$ is a suffix of $v$} \quad \Longleftrightarrow \quad
	\text{$\exists w \in A^\star$ such that $w u = v$,}\\
	& \text{$u$ is a prefix of $v$} \quad \Longleftrightarrow \quad
	\text{$\exists w \in A^\star$ such that $u w = v$,}\\
	& \text{$u$ is a factor of $v$} \quad \Longleftrightarrow \quad
	\text{$\exists w_1, w_2 \in A^\star$ such that $w_1 u w_2 = v$.}
\end{split}
\end{equation*}
A \defn{prefix code} $\mathcal{C}$ of $A^+$ (or over $A$) is a subset $\mathcal{C} \subseteq A^+$ so that all elements 
in $\mathcal{C}$ are pairwise incomparable in the prefix order~\cite{BPR.2010}.

\begin{definition} \cite[Proposition~3.5.4]{BPR.2010}
A prefix code $\mathcal{S} \subset A^+$ is a \defn{semaphore code} if $A \mathcal{S} \subseteq \mathcal{S} A^\star$.
\end{definition}

In other words, a semaphore code is a prefix code $\mathcal{S}$ over $A$ for which there is a left action in
the following sense:
\begin{equation}
\label{equation.semaphore}
\begin{split}
&\text{If $u \in \mathcal{S} \subseteq A^+$ and $a \in A$, then $a u$ has a prefix in $\mathcal{S}$ 
(and hence a unique prefix of $a u$).} \\
&\text{The left action $a.u$ is the prefix of $a u$ that is in $\mathcal{S}$.}
\end{split}
\end{equation}

Semaphore codes over $A$ are inherently related to ideals of $A^+$~\cite[Proposition 4.3]{RSS.2016}. 
Given an ideal $I \subseteq A^+$ we construct a semaphore code as follows. Given $u = a_1 a_2  \ldots a_j \in A^+$, check 
whether $u$ is in $I$. If $u \not \in I$, ignore $u$. If $u \in I$, we find the (necessarily unique) index $1\leqslant i \leqslant j$ 
such that $a_1 \ldots a_{i-1} \not \in I$, but $a_1 \ldots a_i \in I$. Then $a_1 \ldots a_i$ is a code word and the set of all 
such words forms the semaphore code $\mathcal{S}=:I \beta_\ell$. Conversely, given a semaphore code 
$\mathcal{S}$, the corresponding ideal is obtained as $\mathcal{S} A^\star$. This yields a bijection
\begin{equation}
\label{equation.I semaphore}
	I \longleftrightarrow I \beta_\ell
\end{equation}
between ideals $I \subseteq A^+$ and semaphore codes $\mathcal{S}$ over $A$.

Using the left action in~\eqref{equation.semaphore}, we can define a Markov chain $\mathcal{M}^{\mathcal{S}}$
on the semaphore code $\mathcal{S}$. We transition $s \stackrel{a}{\longrightarrow} s'$ for $s,s' \in \mathcal{S}$
and $a\in A$ with probability $x_a$, where $s'=a.s$. The stationary distribution for the Markov chain 
$\mathcal{M}^{\mathcal{S}}$ was computed in~\cite[Theorem 8.1]{RSS.2016}.

\begin{proposition} \cite{RSS.2016}
\label{proposition.stationary semaphore}
The stationary distribution of the Markov chain $\mathcal{M}^{\mathcal{S}}$ is
\[
	\Psi^{\mathcal{S}}_s = \prod_{a \in s} x_a \qquad \text{for all $s \in \mathcal{S}$.}
\]
\end{proposition}

Given a finite semigroup $S$ with generators $A$, recall $\varphi \colon (A^+,A) \to (S,A)$ as defined 
in~\eqref{equation.phi}.

\begin{definition}
\label{definition.ideal}
Let $(S,A)$ be a finite semigroup $S$ with generators $A$ and $I \subseteq S$ an ideal. Define
\[
	\Ideal(S,A,I) := \varphi^{-1}(I) \subseteq A^+.
\]
In particular, let $\Ideal(S,A) := \Ideal(S,A,K(S))$, where recall that $K(S)$ is the minimal ideal in $S$.
\end{definition}

\begin{lemma}
$\Ideal(S,A,I)$ is an ideal in $A^+$.
\end{lemma}

\begin{proof}
Let $u,v \in A^+$ and $w \in \Ideal(S,A,I)$. Then $[uwv]_S = [u]_S [w]_S [v]_S \in I$ since $I \subseteq S$ is an ideal.
\end{proof}

\begin{definition}
\label{definition.S}
Let $\mathcal{S}(S,A)$ be the semaphore code associated to $\Ideal(S,A)$ according to~\eqref{equation.I semaphore}.
\end{definition}

We may now define the Markov chain $\mathcal{M}^{\mathcal{S}}(S,A)$ for the finite $A$-semigroup $(S,A)$
as the Markov chain on the semaphore code $\mathcal{S}(S,A)$. A semigroup $S$ is \defn{left zero} 
if $xy=x$ for all $x,y\in S$. When $K(S)$ is left zero, we can obtain the stationary distribution of
$\mathcal{M}(S,A)$ from the stationary distribution $\mathcal{M}^{\mathcal{S}}(S,A)$ by lumping.

\begin{theorem}
\label{theorem.semaphore lumping}
If $K(S)$ is left zero, then the finite Markov chain $\mathcal{M}(S,A)$ is a lumping of $\mathcal{M}^{\mathcal{S}}(S,A)$,
where $\mathcal{S}$ is the semaphore code associated to $K(S)$.
Furthermore,
\[
	\Psi_w = \sum_{\substack{s \in \mathcal{S}\\ [s]_S = w}} \Psi^{\mathcal{S}}_s \qquad \text{for $w \in K(S)$.}
\]
\end{theorem}

\begin{proof}
We need to check that~\eqref{equation.lumping} holds, where $\mathcal{T}$ is the transition matrix of
$\mathcal{M}^{\mathcal{S}}$. In particular
\begin{equation}
\label{equation.lumping condition}
	\sum_{t, [t]_S=w} \mathcal{T}_{t,s} 
	= \sum_{t, [t]_S=w} \mathcal{T}_{t,s'}
\end{equation}
for all $w \in K(S)$ and $s,s' \in \mathcal{S}$ such that $[s]_S = [s']_S$. Recall from~\eqref{equation.transition matrix} 
that $\mathcal{T}_{t,s} = \sum_{a \in A, t=a.s} x_a$. Since $S$ is a semigroup, we have $[as]_S = [as']_S$ for all
$a\in A$. By the definition of a semaphore, $as$ (resp. $as'$) also has a prefix $a.s$ (resp. $a.s'$) in $K(S)$ under 
$\varphi$. Since $K(S)$ is left zero, this implies that $[a.s]_S = [as]_S$ and similarly $[a.s']_S = [as']_S$. Hence
$[a.s]_S = [a.s']_S$, which implies~\eqref{equation.lumping condition}. The stationary distribution of the lumped Markov 
chain is obtained from the stationary distribution of the unlumped Markov chain by summing over all states in an 
equivalence class. This proves the claim.
\end{proof}

\subsection{The Karnofsky--Rhodes expansion}
\label{section.KR}

To compute explicit expressions for the stationary distributions of Markov chains on finite semigroups, we need the
\defn{Karnofsky--Rhodes expansion}~\cite{Elston.1999} of the right Cayley graph $\mathsf{RCay}(S,A)$.
See also~\cite[Definition~4.15]{MRS.2011} and~\cite[Section 3.4]{MSS.2015}. In addition, we will require
the McCammond expansion~\cite{MRS.2011}, which is discussed in the next section.

\begin{definition}[Karnofksy--Rhodes expansion]
The \defn{Karnofsky--Rhodes expansion} $\mathsf{KR}(S,A)$ is obtained as follows. Start with the right Cayley graph 
$\mathsf{RCay}(A^+,A)$. Identify two paths in $\mathsf{RCay}(A^+,A)$ 
\begin{equation*}
	p := \left( \mathbbm{1} \stackrel{a_1}{\longrightarrow} v_1 \stackrel{a_2}{\longrightarrow} \cdots \stackrel{a_\ell}
	{\longrightarrow} v_\ell \right)
	\quad \text{and} \quad 
	p' := \left( \mathbbm{1} \stackrel{a'_1}{\longrightarrow} v'_1 \stackrel{a'_2}{\longrightarrow} \cdots 
	\stackrel{a'_{\ell'}}{\longrightarrow} v'_{\ell'} \right)
\end{equation*}
in $\mathsf{KR}(S,A)$ if and only if the corresponding paths in $\mathsf{RCay}(S,A)$
\begin{equation*}
	[p]_S := \left( \mathbbm{1} \stackrel{a_1}{\longrightarrow} [w_1]_S \stackrel{a_2}{\longrightarrow} \cdots 
	\stackrel{a_\ell}{\longrightarrow} [w_\ell]_S \right)
	\quad \text{and} \quad 
	[p']_S := \left( \mathbbm{1} \stackrel{a'_1}{\longrightarrow} [w'_1]_S \stackrel{a'_2}{\longrightarrow} \cdots 
	\stackrel{a'_{\ell'}}{\longrightarrow} [w'_{\ell'}]_S \right),
\end{equation*}
where $w_i=a_1 a_2 \ldots a_i$ and $w_i' = a_1' a_2' \ldots a'_i$, end at the same vertex $[w_\ell]_S = [w'_{\ell'}]_S$ 
and in addition the set of transition edges of $[p]_S$ and $[p']_S$ in $\mathsf{RCay}(S,A)$ is equal. 
\end{definition}

An example for $\mathsf{KR}(S,A)$ is given in Figure~\ref{figure.KR}. In this figure, the paths $a^2b$ and $aba$
are equal because they end in the same vertex when projected onto $S$ and they share the same transition edge, 
which is the first $a$. On the other hand, the paths $ab$ and $ba$ are distinct because for the first path the transition
edge is the first $a$ and for the second path the transition edge is the first $b$.

\begin{figure}[t]
\begin{center}
\begin{tikzpicture}[auto]
\node (A) at (0, 0) {$\mathbbm{1}$};
\node (B) at (-1,-1) {$a$};
\node(C) at (1,-1) {$b$};
\node(D) at (-1,-2) {$ab$};
\node(E) at (1,-2) {$ba$};
\node(F) at (-2.5,-2) {$a^2$};
\node(G) at (2.5,-2) {$b^2$};
\node(H) at (-1,-3) {$a^2b=aba$};
\node(I) at (1,-3) {$bab=b^2a$};
\draw[edge,blue,thick] (A) -- (B) node[midway, above] {$a$\;};
\draw[edge,thick,blue] (A) -- (C) node[midway, above] {\;$b$};
\draw[edge,thick] (B) -- (F) node[midway,left] {$a$\;};
\draw[edge,thick] (F) -- (B);
\draw[edge,thick] (B) -- (D) node[midway,right] {$b$};
\draw[edge,thick] (D) -- (B);
\draw[edge,thick] (F) -- (H) node[midway,left] {$b$\;};
\draw[edge,thick] (H) -- (F);
\draw[edge,thick] (D) -- (H) node[midway,right] {\;$a$};
\draw[edge,thick] (H) -- (D);
\draw[edge,thick] (C) -- (E) node[midway,left] {$a$\;};
\draw[edge,thick] (E) -- (C);
\draw[edge,thick] (E) -- (I) node[midway,left] {$b$\;};
\draw[edge,thick] (I) -- (E);
\draw[edge,thick] (C) -- (G) node[midway,right] {\;$b$};
\draw[edge,thick] (G) -- (C);
\draw[edge,thick] (G) -- (I) node[midway,right] {\;$a$};
\draw[edge,thick] (I) -- (G);
\end{tikzpicture}
\end{center}
\caption{\label{figure.KR} The Karnofsky--Rhodes expansion $\mathsf{KR}(S,A)$ of the right Cayley graph of
Figure~\ref{figure.right cayley}.}
\end{figure}
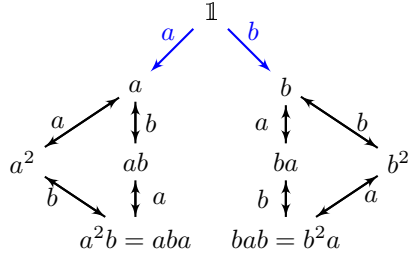

\begin{example}
\label{example.01}
Take the semigroup $S=\{0,1\}$ under multiplication with generators $A=\{0,1\}$, where
$0\cdot x=x\cdot 0 = 0$ and $1\cdot x = x \cdot 1 = x$ for all $x \in S$.
The right Cayley graph and its Karnofsky--Rhodes expansion are given by
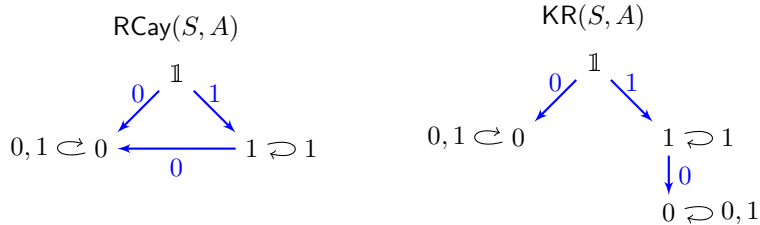
\begin{center}
\begin{tikzpicture}[auto]
\node (EX) at (0,0.6) {$\mathsf{RCay}(S,A)$};
\node (A) at (0, 0) {$\mathbbm{1}$};
\node (B) at (-1,-1) {$0$};
\node(C) at (1,-1) {$1$};
\node (D) at (1,-2) {};
\draw[edge,blue,thick] (A) -- (B) node[midway, above] {$0$};
\draw[edge,thick,blue] (A) -- (C) node[midway, above] {$1$};
\draw[edge,thick,blue] (C) -- (B) node[midway, below] {$0$};
\path
	(B) edge [loop left] node {$0,1$} (B)
	(C) edge [loop right] node {$1$} (C);
\end{tikzpicture}
\hspace{1cm}
\begin{tikzpicture}[auto]
\node (EX) at (0,0.6) {$\mathsf{KR}(S,A)$};
\node (A) at (0, 0) {$\mathbbm{1}$};
\node (B) at (-1,-1) {$0$};
\node(C) at (1,-1) {$1$};
\node(D) at (1,-2) {$0$};
\draw[edge,blue,thick] (A) -- (B) node[midway, above] {$0$};
\draw[edge,thick,blue] (A) -- (C) node[midway, above] {$1$};
\draw[edge,thick,blue] (C) -- (D) node[midway, right] {$0$};
\path
	(B) edge [loop left] node {$0,1$} (B)
	(C) edge [loop right] node {$1$} (C)
	(D) edge [loop right] node {$0,1$} (D);
\end{tikzpicture}
\end{center}
\end{example}

\begin{proposition}
\label{proposition.KR Cayley}
$\mathsf{KR}(S,A)$ is the right Cayley graph of a semigroup, also denoted by $\mathsf{KR}(S,A)$.
\end{proposition}

\begin{proof}
Since the graph $\mathsf{KR}(S,A)$ is constructed from the right Cayley graph $\mathsf{RCay}(A^+,A)$,
many of the properties of right Cayley graphs are automatically satisfied. That is, $\mathsf{KR}(S,A)$ is deterministic 
and complete, the root $\mathbbm{1}$ is not the endpoint of any edge, and every vertex is accessible from 
$\mathbbm{1}$. In addition to this, we need to check that if two paths $p=p_1 p_2 \ldots p_\ell$ and $q=q_1 q_2 \ldots q_k$ 
(written in terms of their edge labels $p_i$ and $q_i$) starting at $\mathbbm{1}$ satisfy
$\tau(p)=\tau(q)$, then $\tau(yp)=\tau(yq)$ for any path $y$ in $\mathsf{KR}(S,A)$. Here $yp$ stands for the
path given by the concatenation of the edge labels of $y$ with those of $p$.
The condition $\tau(p)=\tau(q)$ by the definition of the Karnofsky--Rhodes expansion is equivalent to
the conditions that $[p]_S = [q]_S$ and that the set of transition edges in $p$ and $q$ agree. 

Since $\mathsf{RCay}(S,A)$ is a right Cayley graph, we have $[yp]_S = [yq]_S$ for any path $y$. Now suppose 
by contradiction that the transition edges in $yp$ and $yq$ do not agree. Note that a non-transition edge $p_i$ in $p$ 
cannot become a transition edge $p_i$ in $yp$. Hence, without loss of generality, let us assume that there is a transition 
edge $p_i$ in $yp$ that is not a transition edge in $yq$ and that all transition edges among $p_1,\ldots,p_{i-1}$ are also
transition edges in $yq$. Since in $p$ and $q$ all transition edges agree by assumption, let $q_j$ in $q$ be the transition 
edge corresponding to $p_i$ in $p$. In particular, this implies that $\tau(p_1 \ldots p_{i-1}) = \tau(q_1 \ldots q_{j-1})$
and $\tau(p_1 \ldots p_i) = \tau(q_1 \ldots q_j)$. But this in turn implies that $v=\tau(yp_1 \ldots p_{i-1}) 
= \tau(yq_1 \ldots q_{j-1})$ and $w= \tau(yp_1 \ldots p_i) = \tau(yq_1 \ldots q_j)$, meaning that the edge between
vertex $v$ and vertex $w$ is the same in $yp$ and $yq$, contradicting the assumption that the edge is a transition
edge in $yp$, but not in $yq$. Hence $\mathsf{KR}(S,A)$ is a right Cayley graph.
\end{proof}

Since by Proposition~\ref{proposition.KR Cayley} $\mathsf{KR}(S,A)$ is the right Cayley graph
of a semigroup, we can consider the corresponding Markov chain $\mathcal{M}(\mathsf{KR}(S,A))$.
The Markov chain $\mathcal{M}(S,A)$ can be obtained from $\mathcal{M}(\mathsf{KR}(S,A))$ by the 
projection $w \mapsto [w]_S$ for $w \in \mathsf{KR}(S,A)$ since both $(S,A)$ and $\mathsf{KR}(S,A)$
are semigroups.

\begin{proposition}
The Karnofsky--Rhodes expansion of $\mathsf{KR}(S,A)$ is stable, that is
\[
	\mathsf{KR}(\mathsf{KR}(S,A),A) = \mathsf{KR}(S,A).
\]
\end{proposition}

\begin{proof}
This is clear from the definition, since the set of transition edges does not change 
from $\mathsf{KR}(S,A)$ to $\mathsf{KR}(\mathsf{KR}(S,A),A)$.
\end{proof}

For additional properties of the Karnofsky--Rhodes expansion, see~\cite{Elston.1999, MRS.2011,MSS.2015}.

\subsection{The McCammond expansion}
\label{section.mccammond}

The McCammond expansion~\cite{MRS.2011} of a rooted graph is intimately related to the 
unique simple path property.

\begin{definition}[Unique simple path property]
A rooted graph $(\Gamma,r)$ has the \defn{unique simple path property} if for each vertex $v\in V(\Gamma)$ there is
a unique simple path from the root $r$ to $v$.
\end{definition}

As proven in~\cite[Proposition 2.32]{MRS.2011}, the unique simple path property is equivalent to $(\Gamma,r)$
admitting a unique directed spanning tree $\mathsf{T}$. Note that the unique simple path property not only depends 
on the graph $\Gamma$, but also on the chosen root $r$. In this paper, we always choose $r= \mathbbm{1}$.

It was established in~\cite[Section 2.7]{MRS.2011} that every rooted graph $(\Gamma,r)$ has a universal simple
cover, which has the unique simple path property.

\begin{definition}[McCammond expansion]
\label{definition.mccammond}
For a rooted graph $(\Gamma,r)$, define its \defn{McCammond expansion} $(\Gamma^{\mathsf{Mc}},r)$ as the graph with
\begin{equation*}
\begin{split}
	V(\Gamma^{\mathsf{Mc}}) &= \mathsf{Simple}(\Gamma,r),\\
	E(\Gamma^{\mathsf{Mc}}) &= \{(p,a,q) \in V(\Gamma^{\mathsf{Mc}}) \times A \times V(\Gamma^{\mathsf{Mc}}) \mid
	(\tau(p), a, \tau(q)) \in E(\Gamma),\\
	& \qquad \qquad  \qquad \ell(q) = \ell(p)+1 \text{ or } (q\subseteq p \text{ and } \ell(q)\leqslant \ell(p))\}.
\end{split}
\end{equation*}
\end{definition}

Note that by definition there are two types of edges $(p,a,q) \in E(\Gamma^{\mathsf{Mc}})$: either $\ell(q)=\ell(p)+1$
or $\ell(q) \leqslant \ell(p)$ as paths in $\mathsf{Simple}(\Gamma,r)$. The spanning tree $\mathsf{T}$ has vertex set
$V(\Gamma^{\mathsf{Mc}})$ and only those edges $(p,a,q) \in E(\Gamma^{\mathsf{Mc}})$ such that $\ell(q)=\ell(p)+1$.

From now on choose $r = \mathbbm{1}$. The simple path 
\[
	\mathbbm{1} \stackrel{a_1}{\longrightarrow} v_1 \stackrel{a_2}{\longrightarrow} \cdots \stackrel{a_\ell}{\longrightarrow} 
	v_\ell
\]
in $\mathsf{Simple}(\Gamma,\mathbbm{1})$ is naturally indexed by the word $a_1 a_2 \ldots a_\ell$. We will use this labeling
for the McCammond expansion of $\mathsf{KR}(S,A)$. In particular, if $a_1 a_2 \ldots a_\ell \in 
\mathsf{Simple}(\Gamma,\mathbbm{1})$ and $a_1 a_2 \ldots a_\ell a \in \mathsf{Simple}(\Gamma,\mathbbm{1})$, then 
the edge $a_1 a_2 \ldots a_\ell \stackrel{a}{\longrightarrow} a_1 a_2 \ldots a_\ell a$ is in the spanning tree $\mathsf{T}$.
Otherwise we have $a_1 a_2 \ldots a_\ell \stackrel{a}{\longrightarrow} a_1 a_2 \ldots a_k$ for some unique $1\leqslant k < \ell$.
Thus under the right action of $a \in A$ on $a_1 a_2 \ldots a_\ell$, we either move forward in the spanning tree or
fall backwards somewhere on the unique geodesic from $\mathbbm{1}$ to $a_1 a_2 \ldots a_\ell$, 
but staying in the same $\mathscr{R}$-class. An example of a McCammond expansion of a Karnofsky--Rhodes graph is 
given in Figure~\ref{figure.mccammond}.

For a non-simple path in $(\Gamma^{\mathsf{Mc}},\mathbbm{1})$, we can remove loops; it does not matter in which order
these loops are removed. This is also known as the Church--Rosser property~\cite{Church.Rosser.1936} or a
Knuth--Bendix rewriting system. This is proved in~\cite{MRS.2011}.

We denote the McCammond expansion of an $A$-semigroup $(S,A)$ by $\mathsf{Mc}(S,A)$, which is the McCammond
expansion of its right Cayley graph. The McCammond expansion of the right Cayley graph is in general not a right 
Cayley graph, see Remark~\ref{remark.counter example} below. This makes random walks on semigroups more 
difficult to understand. It is however true that the McCammond expansion is stable under repeated McCammond expansions
if the root is unchanged.

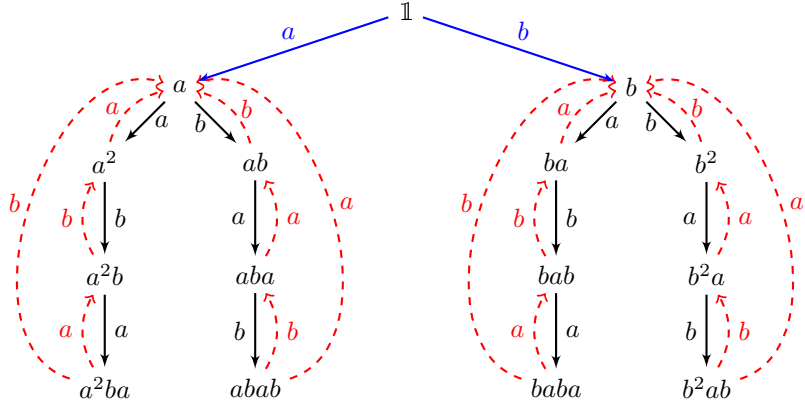
\begin{figure}[t]
\begin{center}
\begin{tikzpicture}[auto]
\node (A) at (0, 0) {$\mathbbm{1}$};
\node (B) at (-3,-1) {$a$};
\node(C) at (3,-1) {$b$};
\node(D) at (-2,-2) {$ab$};
\node(E) at (2,-2) {$ba$};
\node(F) at (-4,-2) {$a^2$};
\node(FF) at (-4,-3.5) {$a^2b$};
\node(FFF) at (-4,-5) {$a^2ba$};
\node(G) at (4,-2) {$b^2$};
\node(H) at (-2,-3.5) {$aba$};
\node(HH) at (-2,-5) {$abab$};
\node(EE) at (2,-3.5) {$bab$};
\node(EEE) at (2,-5) {$baba$};
\node(GG) at (4,-3.5) {$b^2a$};
\node(GGG) at (4,-5) {$b^2ab$};
\draw[edge,blue,thick] (A) -- (B) node[midway, above] {$a$\;};
\draw[edge,thick,blue] (A) -- (C) node[midway, above] {\;$b$};
\draw[edge,thick] (B) -- (F) node[midway,right] {$a$};
\path (F) edge[->,thick, red,dashed, bend left=30] node[midway,left] {$a$} (B);
\draw[edge,thick] (F) -- (FF) node[midway,right] {$b$};
\path (FF) edge[->,thick, red, dashed, bend left=30] node[midway,left] {$b$} (F);
\draw[edge,thick] (FF) -- (FFF) node[midway,right] {$a$};
\path (FFF) edge[->,thick, red, dashed, bend left=30] node[midway,left] {$a$} (FF);
\path (FFF) edge[->,thick, red,dashed, bend left=90] node[midway,left] {$b$} (B);
\draw[edge,thick] (B) -- (D) node[midway,left] {$b$};
\path (D) edge[->,thick, red, dashed, bend right=30] node[midway,right] {$b$} (B);
\draw[edge,thick] (D) -- (H) node[midway,left] {$a$};
\path (H) edge[->,thick, red, dashed, bend right=30] node[midway,right] {$a$} (D);
\draw[edge,thick] (H) -- (HH) node[midway,left] {$b$};
\path (HH) edge[->,thick, red, dashed, bend right=30] node[midway,right] {$b$} (H);
\path (HH) edge[->,thick, red,dashed, bend right=90] node[midway,right] {$a$} (B);
\draw[edge,thick] (C) -- (E) node[midway,right] {$a$};
\path (E) edge[->,thick, red,dashed, bend left=30] node[midway,left] {$a$} (C);
\draw[edge,thick] (E) -- (EE) node[midway,right] {$b$};
\path (EE) edge[->,thick, red,dashed, bend left=30] node[midway,left] {$b$} (E);
\draw[edge,thick] (EE) -- (EEE) node[midway,right] {$a$};
\path (EEE) edge[->,thick, red,dashed, bend left=30] node[midway,left] {$a$} (EE);
\draw[edge,thick] (C) -- (G) node[midway,left] {$b$};
\path (G) edge[->,thick, red,dashed, bend right=30] node[midway,right] {$b$} (C);
\draw[edge,thick] (G) -- (GG) node[midway,left] {$a$};
\path (GG) edge[->,thick, red,dashed, bend right=30] node[midway,right] {$a$} (G);
\draw[edge,thick] (GG) -- (GGG) node[midway,left] {$b$};
\path (GGG) edge[->,thick, red,dashed, bend right=30] node[midway,right] {$b$} (GG);
\path (EEE) edge[->,thick, red,dashed, bend left=90] node[midway,left] {$b$} (C);
\path (GGG) edge[->,thick, red,dashed, bend right=90] node[midway,right] {$a$} (C);
\end{tikzpicture}
\end{center}
\caption{\label{figure.mccammond} The McCammond expansion of $(\Gamma,\mathbbm{1}) = \mathsf{KR}(S,A)$ 
of Figure~\ref{figure.KR}.
The edges $(p,a,q) \in E(\Gamma^{\mathsf{Mc}})$ with $\ell(q)=\ell(p)+1$ are solid, whereas the edges
with $\ell(q) \leqslant \ell(p)$ are dashed and red.}
\end{figure}

\begin{remark}
\label{remark.counter example}
The McCammond expansion of a right Cayley graph is not always a right Cayley graph itself. Let $S$ be the semigroup
generated by the elements in $A=\{a_1,a_2,a_3,c\}$ which act on the states $Q=\{0,1,2,3,\square\}$ as follows:
\[
	a_1 \colon \begin{array}{l} 0 \mapsto 1 \\ 3 \mapsto 2 \end{array}, \qquad
	a_2 \colon \begin{array}{l} 1 \mapsto 2 \\ 2 \mapsto 1 \end{array}, \qquad
	a_3 \colon \begin{array}{l} 2 \mapsto 3 \\ 1 \mapsto 3 \end{array}, \qquad
	c \colon \begin{array}{l} 0 \mapsto 0 \\ 1 \mapsto 0 \\ 2 \mapsto 3 \\ 3 \mapsto 0 \end{array},
\]
and otherwise $q \mapsto \square$. For part of the right Cayley graph of $S$, see Figure~\ref{figure.counter example}.
Then $a_1 a_2 a_3$ and $a_1 a_2 a_3 a_1 a_2 a_3$ end at the same vertex in both $\mathsf{RCay}(S,A)$ and
$\mathsf{Mc} \circ \mathsf{KR}(S,A)$, that is $\tau(a_1 a_2 a_3) = \tau(a_1 a_2 a_3 a_1 a_2 a_3)$.
However, multiplying on the left by $c$, we obtain
\[
	c a_1 a_2 a_3  = \tau(c a_1 a_2 a_3) \not = 
	\tau(c a_1 a_2 a_3 a_1 a_2 a_3) = c a_1 a_3
\]
in $\mathsf{Mc} \circ \mathsf{KR}(S,A)$, as can be seen from the right picture in Figure~\ref{figure.counter example}.
\end{remark}

\begin{figure}[t]
\begin{center}
\begin{tikzpicture}[auto]
\node (A) at (0, 0) {$\mathbbm{1}$};
\node (C) at (-1.5,-1.5) {$\bullet$};
\node (D) at (-1.5,-3) {$\bullet$};
\node (E) at (-1.5,-4.5) {$\bullet$};
\node (F) at (-1.5,-6) {$\bullet$};
\node (G) at (-1.5,-7.5) {$\bullet$};
\node (Bp) at (1.5,-1.5) {$\bullet$};
\node (Cp) at (1.5,-3) {$\bullet$};
\node (Dp) at (1.5,-4.5) {$\bullet$};
\node (Ep) at (1.5,-6) {$\bullet$};
\draw[edge,blue,thick] (A) -- (C) node[midway, above] {$a_1\;$};
\draw[edge,blue,thick] (A) -- (Bp) node[midway, above] {$c$};
\path (Bp) edge[->,thick] node[midway,right] {$a_1$} (Cp);
\path (Cp) edge[->,thick] node[midway,right] {$a_2$} (Dp);
\path (Dp) edge[->,thick] node[midway,right] {$a_3$} (Ep);
\path (Cp) edge[->,thick, bend left = 60] node[midway,right] {$a_3$} (Ep);
\path (Ep) edge[->,thick, bend left = 60] node[midway,left] {$a_1$} (Dp);
\path (Dp) edge[->,thick, bend left = 60] node[midway,left] {$a_2$} (Cp);
\path (C) edge[->,thick] node[midway,left] {$a_2$} (D);
\path (D) edge[->,thick,blue] node[midway,left] {$a_3$} (E);
\path (C) edge[->,thick, bend right = 60,blue] node[midway,left] {$a_3$} (E);
\path (E) edge[->,thick] node[midway,left] {$a_1$} (F);
\path (F) edge[->,thick] node[midway,left] {$a_2$} (G);
\path (G) edge[->,thick, bend left = 60] node[midway,left] {$a_3$} (E);
\path (G) edge[->,thick, bend right = 60] node[midway,right] {$a_2$} (F);
\path (F) edge[->,thick, bend right = 60] node[midway,right] {$a_3$} (E);
\path (D) edge[->,thick, bend right = 60] node[midway,right] {$a_2$} (C);
\end{tikzpicture}
\begin{tikzpicture}[auto]
\node (A) at (0, 0) {$\mathbbm{1}$};
\node (C) at (-1.5,-1.5) {$\bullet$};
\node (D) at (-1.5,-3) {$\bullet$};
\node (E) at (-1.5,-4.5) {$\bullet$};
\node (F) at (-1.5,-6) {$\bullet$};
\node (G) at (-1.5,-7.5) {$\bullet$};
\node (Bp) at (1.5,-1.5) {$\bullet$};
\node (Cp) at (1.5,-3) {$\bullet$};
\node (Dpp) at (3,-4.5) {$\bullet$};
\node (Dp) at (1.5,-4.5) {$\bullet$};
\node (Ep) at (1.5,-6) {$\bullet$};
\node (Epp) at (3,-6) {$\bullet$};
\node (N) at (-3,-3) {$\bullet$};
\node (Np) at (-3,-3.5) {$\vdots$};
\draw[edge,blue,thick] (A) -- (C) node[midway, above] {$a_1\;$};
\draw[edge,blue,thick] (A) -- (Bp) node[midway, above] {$c$};
\path (Bp) edge[->,thick] node[midway,right] {$a_1$} (Cp);
\path (Cp) edge[->,thick] node[midway,right] {$a_2$} (Dp);
\path (Dp) edge[->,thick] node[midway,right] {$a_3$} (Ep);
\path (Cp) edge[->,thick] node[midway,right] {$a_3$} (Dpp);
\path (Dpp) edge[->,thick] node[midway,right] {$a_1$} (Epp);
\path (Ep) edge[->,thick, bend left = 40, red, dashed] node[midway,left] {$a_1$} (Dp);
\path (Dp) edge[->,thick, bend left = 40, red, dashed] node[midway,left] {$a_2$} (Cp);
\path (C) edge[->,thick] node[midway,left] {$a_2$} (D);
\path (D) edge[->,thick,blue] node[midway,left] {$a_3$} (E);
\path (C) edge[->,thick, blue] node[midway,left] {$a_3$} (N);
\path (E) edge[->,thick] node[midway,left] {$a_1$} (F);
\path (F) edge[->,thick] node[midway,left] {$a_2$} (G);
\path (G) edge[->,thick, bend left = 60, red, dashed] node[midway,left] {$a_3$} (E);
\path (G) edge[->,thick, bend right = 60, red, dashed] node[midway,right] {$a_2$} (F);
\path (F) edge[->,thick, bend right = 60, red, dashed] node[midway,right] {$a_3$} (E);
\path (D) edge[->,thick, bend right = 60, red, dashed] node[midway,right] {$a_2$} (C);
\path (Epp) edge[->,thick, bend left = 40, red, dashed] node[midway,left] {$a_3$} (Dpp);
\path (Epp) edge[->,thick, bend right = 80, red, dashed] node[midway,right] {$a_2$} (Cp);
\end{tikzpicture}
\caption{\textbf{Left:} Part of the right Cayley graph of the semigroup $(S,A)$ in Remark~\ref{remark.counter example}; 
transition edges are blue; only one $c$ edge is drawn. \textbf{Right:} Part of the expansion
$\mathsf{Mc} \circ \mathsf{KR}(S,A)$.
\label{figure.counter example}}
\end{center}
\end{figure}
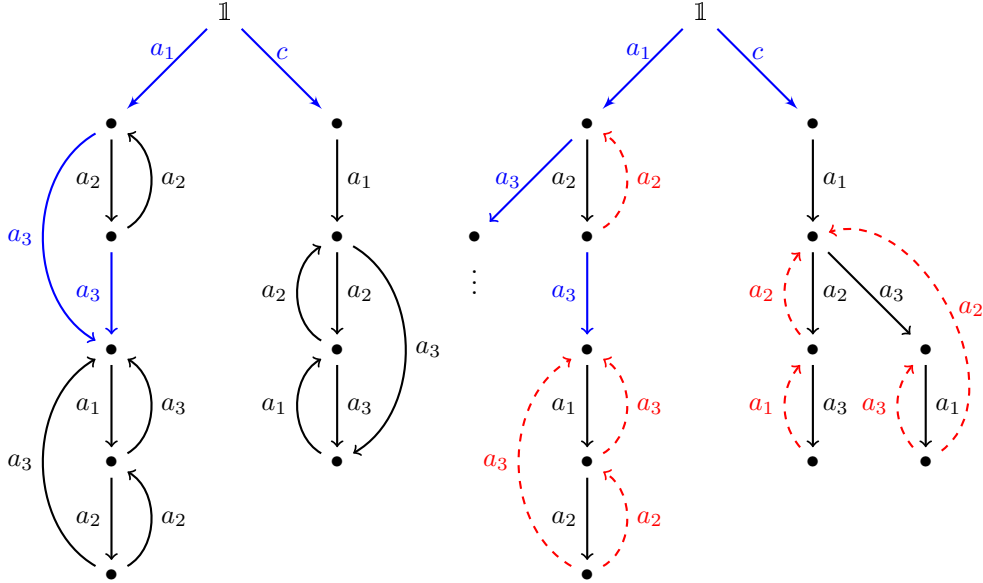

\subsection{Normal forms}
\label{section.normal}

Consider $\Gamma(S,A):=\mathsf{Mc} \circ \mathsf{KR}(S,A)$. The non-empty paths in $\Gamma(S,A)$ starting at
$\mathbbm{1}$ are naturally labeled by elements in $A^+$. The elements in the semaphore code $\mathcal{S}(S,A)$ 
of Definition~\ref{definition.S} are in natural correspondence with the paths $p = a_1 a_2 \cdots a_\ell \in A^+$ in the 
rooted graph $\Gamma(S,A)$ starting at $\mathbbm{1}$ such that $[a_1 a_2 \cdots a_\ell]_S \in K(S)$, but
$[a_1 a_2 \cdots a_{\ell-1}]_S \not \in K(S)$. The paths in $\mathcal{S}(S,A)$, considered in $\Gamma(S,A)$, do not 
necessarily have to be simple, that  is, they can contain loops.

Recall from Definition~\ref{definition.mccammond}, that the vertices in $\Gamma(S,A)$ are simple paths (that is,
paths without loops) in the Karnofsky--Rhodes expansion of the right Cayley graph of $(S,A)$. Hence given
an arbitrary path $p$ in $\Gamma(S,A)$, its endpoint $\tau(p)$ is a vertex in $\Gamma(S,A)$ and hence a 
simple path in $\mathsf{KR}(S,A)$, which is $p$ read in $\Gamma(S,A)$ with ``loops stripped away''.

\begin{definition}
\label{definition.normal}
Define the set of \defn{normal forms}
\[
	\mathcal{N}(S,A) = \{ \tau(p) \mid p \in \mathcal{S}(S,A)\}.
\]
\end{definition}

Note that the normal forms in $\mathcal{N}(S,A)$ are precisely the elements in $\mathcal{S}(S,A)$ 
(as considered as elements in $\Gamma(S,A)$) without loops.
In other words, they are the shortest simple paths in $\mathsf{KR}(S,A)$ from the root $\mathbbm{1}$ to the ideal.

\begin{remark}
We can construct a new graph from $\Gamma(S,A)$ by contracting each $\mathscr{R}$-class to a vertex. 
The remaining edges correspond to transition edges in $\mathsf{RCay}(S,A)$. The resulting graph is a tree. Given an
ideal $I\subseteq S$, the vertices $v$ in $\Gamma(S,A)$ such that $[v]_S \in I$ project to a lower set in this tree.
\end{remark}

As outlined in Section~\ref{section.semaphore}, there is a Markov chain $\mathcal{M}^{\mathcal{S}(S,A)}$
associated to the semaphore code $\mathcal{S}(S,A)$.

\begin{example}
\label{example.markov}
Consider $S=Z_2 \times \{0,1\}$ with generators $A=\{a,b\}$ with $a=(z,0)$ and $b=(z,1)$, where
$z$ is the generator of $Z_2$ with $z^2=1$ and the operation on $\{0,1\}$ is multiplication. The right Cayley graph and 
its Karnofsky--Rhodes expansion are given by
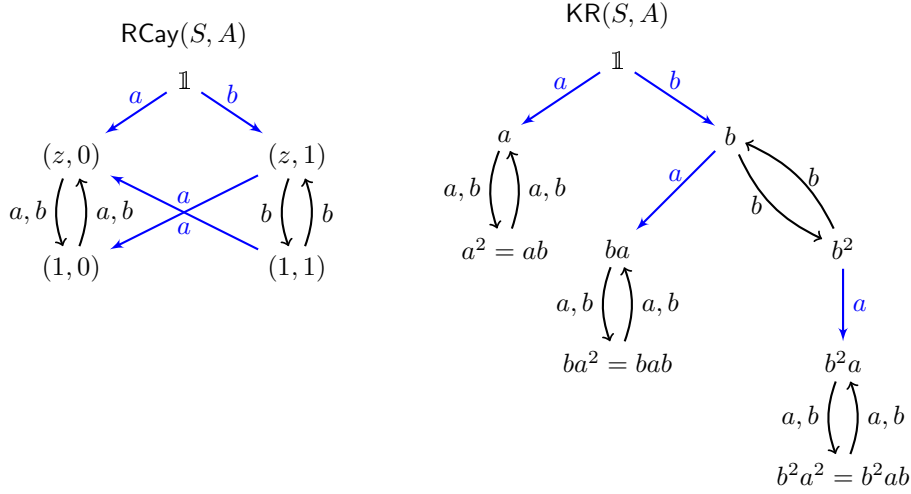
\begin{center}
\raisebox{2.7cm}{
\begin{tikzpicture}
\node (EX) at (0,0.6) {$\mathsf{RCay}(S,A)$};
\node (A) at (0, 0) {$\mathbbm{1}$};
\node (B) at (-1.5,-1) {$(z,0)$};
\node(C) at (1.5,-1) {$(z,1)$};
\node (D) at (-1.5,-2.5) {$(1,0)$};
\node (E) at (1.5,-2.5) {$(1,1)$};
\draw[edge,blue,thick] (A) -- (B) node[midway, above] {$a$};
\draw[edge,thick,blue] (A) -- (C) node[midway, above] {$b$};
\draw[edge,thick,blue] (C) -- (D) node[midway, above] {$a$};
\draw[edge,thick,blue] (E) -- (B) node[midway, below] {$a$};
\path (C) edge[->,thick, bend right = 20] node[midway,left] {$b$} (E);
\path (E) edge[->,thick, bend right = 20] node[midway,right] {$b$} (C);
\path (B) edge[->,thick, bend right = 20] node[midway,left] {$a,b$} (D);
\path (D) edge[->,thick, bend right = 20] node[midway,right] {$a,b$} (B);
\end{tikzpicture}}
\hspace{1cm}
\begin{tikzpicture}
\node (EX) at (0,0.6) {$\mathsf{KR}(S,A)$};
\node (A) at (0, 0) {$\mathbbm{1}$};
\node (B) at (-1.5,-1) {$a$};
\node(C) at (1.5,-1) {$b$};
\node (D) at (-1.5,-2.5) {$a^2=ab$};
\node (E) at (3,-2.5) {$b^2$};
\node (F) at (0,-2.5) {$ba$};
\node (G) at (0,-4) {$ba^2=bab$};
\node (H) at (3,-4) {$b^2a$};
\node (I) at (3,-5.5) {$b^2a^2=b^2ab$};
\draw[edge,blue,thick] (A) -- (B) node[midway, above] {$a$};
\draw[edge,thick,blue] (A) -- (C) node[midway, above] {$b$};
\draw[edge,thick,blue] (C) -- (F) node[midway, above] {$a$};
\draw[edge,thick,blue] (E) -- (H) node[midway, right] {$a$};
\path (C) edge[->,thick, bend right = 20] node[midway,left] {$b$} (E);
\path (E) edge[->,thick, bend right = 20] node[midway,right] {$b$} (C);
\path (B) edge[->,thick, bend right = 20] node[midway,left] {$a,b$} (D);
\path (D) edge[->,thick, bend right = 20] node[midway,right] {$a,b$} (B);
\path (F) edge[->,thick, bend right = 20] node[midway,left] {$a,b$} (G);
\path (G) edge[->,thick, bend right = 20] node[midway,right] {$a,b$} (F);
\path (H) edge[->,thick, bend right = 20] node[midway,left] {$a,b$} (I);
\path (I) edge[->,thick, bend right = 20] node[midway,right] {$a,b$} (H);
\end{tikzpicture}
\end{center}
$\mathsf{KR}(S,A)$ is not stable under $\mathsf{Mc}$ since $\mathsf{KR}(S,A)$ does not have the
unique path property: there are several edges, where $a$ and $b$ act in the same way which is reflected
in the fact that $a^2=ab$, $ba^2=bab$ and $b^2a^2=b^2ab$. The minimal ideal is $K(S)=\{(z,0),(1,0)\}$. 

Let us consider the semigroup $S':= S/K(S)$. Then
\[
	\mathcal{N}(S',A) = \{a,ba,b^2a\} \quad \text{and} \quad
	\mathcal{S}(S',A) = \{a,bb^{2k}a, b^2 b^{2k}a \mid k \geqslant 0\}.
\]
The Markov chain $\mathcal{M}^{\mathcal{S}(S',A)}$ is given by the left action
\[
\begin{aligned}
	a.a &= a, & a.bb^{2k}a &= a, & a.b^2 b^{2k}a &= a,\\
	b.a &= ba, & b.bb^{2k}a &= b^2 b^{2k}a, & b. b^2 b^{2k} a &= b b^{2(k+1)} a.
\end{aligned}
\]
\end{example}

\subsection{Lumping}
\label{section.lumping}

We are now ready to construct the two Markov chains $\mathcal{M}(S,A)$ and $\mathcal{M}(\mathsf{KR}(S,A))$
as projections or lumpings of $\mathcal{M}^{\mathcal{S}(S,A)}$ in the case when the minimal ideal $K(S)$
is left zero. The chain $\mathcal{M}^{\mathcal{S}(S,A)}$ is called a \defn{prelibrary} and
$\mathcal{M}(S,A)$ and $\mathcal{M}(\mathsf{KR}(S,A))$ are called \defn{libraries} when $K(S)$ is left zero.

The lumped chain $\mathcal{M}(\mathsf{KR}(S,A))$ is obtained from $\mathcal{M}^{\mathcal{S}(S,A)}$ 
via the equivalence relation $s \sim s'$ if $[s]_{\mathsf{KR}(S,A)} = [s']_{\mathsf{KR}(S,A)}$ for $s,s'\in \mathcal{S}(S,A)$,
whereas the lumped chain $\mathcal{M}(S,A)$ is obtained from $\mathcal{M}(\mathsf{KR}(S,A))$
via the equivalence relation $w \sim w'$ if $[w]_{(S,A)} = [w']_{(S,A)}$ for $w,w' \in \mathsf{KR}(S,A)$.
In particular, for $w,w'\in \mathsf{KR}(S,A)$ (resp. $(S,A)$) 
we transition $w \stackrel{a}{\longrightarrow} w'$ with probability $x_a$, where $w'=aw$.

\begin{corollary}
\label{corollary.stationary}
If $K(S)$ is left zero, the Markov chain $\mathcal{M}(\mathsf{KR}(S,A))$ is the lumping of 
$\mathcal{M}^{\mathcal{S}(S,A)}$. Furthermore, the stationary distribution of the library
$\mathcal{M}(\mathsf{KR}(S,A))$ is given by
\begin{equation*}
\begin{split}
	\Psi^{\mathsf{KR}(S,A)}_w &= \sum_{\substack{s \in \mathcal{S}(S,A)\\ [s]_{\mathsf{KR}(S,A)}=w}}
	\Psi^{\mathcal{S}(S,A)}_s
	= \sum_{\substack{s \in \mathcal{S}(S,A)\\ [s]_{\mathsf{KR}(S,A)}=w}}\; \prod_{a\in s} x_a
	\qquad \text{for all $w \in K(\mathsf{KR}(S,A))$.}
\end{split}
\end{equation*}
\end{corollary}

\begin{proof}
The second equality follows directly from Proposition~\ref{proposition.stationary semaphore}.
The other statement follows in a similar fashion as in the proof of Theorem~\ref{theorem.semaphore lumping}.
Note also that since $K(S)$ is left zero, we have $K(\mathsf{KR}(S,A))=\{[v]_{\mathsf{KR}(S,A)} \mid v \in \mathcal{N}(S,A)\}$. 
\end{proof}

\begin{remark}
Note that the states $w \in \mathsf{KR}(S,A) \setminus K(\mathsf{KR}(S,A))$ are not recurrent and hence
$\Psi^{\mathsf{KR}(S,A)}_w =0$.
\end{remark}

\begin{remark}
\label{remark.mccammond stable}
If the McCammond expansion of the Karnofsky--Rhodes expansion of the right Cayley graph of $(S,A)$ is
stable, then by Proposition~\ref{proposition.KR Cayley} $\mathsf{Mc}\circ \mathsf{KR}(S,A)= \mathsf{KR}(S,A)$
is a right Cayley graph. In this case Corollary~\ref{corollary.stationary} follows from the fact that 
$\mathsf{Mc} \circ \mathsf{KR}(S,A)$ is a right Cayley graph.
\end{remark}

\begin{example}
\label{example.markov lumped}
Let us continue Example~\ref{example.markov}. The minimal ideal $K(S')$ is trivial and therefore left zero.
Hence by Corollary~\ref{corollary.stationary}, the Markov chain $\mathcal{M}(\mathsf{KR}(S',A))$ is obtained from 
$\mathcal{M}^{\mathcal{S}(S',A)}$ by lumping and is depicted in Figure~\ref{figure.markov}.
\end{example}

\begin{figure}
\begin{tikzpicture}[->,>=stealth',shorten >=1pt,auto,node distance=2.8cm,
                    semithick]
  \tikzstyle{every state}=[fill=red,draw=none,text=white]

  \node[state]         (A) {$a$};
  \node[state]         (B) [right of=A] {$ba$};
  \node[state]         (C) [right of=B] {$b^2a$};

  \path (A) edge [loop left] node {$a$} (A)
                 edge [bend left] node {$b$} (B)
            (B) edge [bend left] node {$b$} (C)
                  edge[bend left] node {$a$} (A)
            (C) edge [bend left] node {$b$} (B)
                  edge [bend left=45] node {$a$} (A);             
\end{tikzpicture}
\caption{Markov chain $\mathcal{M}(\mathsf{KR}(S',A))$ of Example~\ref{example.markov lumped}.
\label{figure.markov}}
\end{figure}

To compute the stationary distributions of Corollary~\ref{corollary.stationary} more explicitly, we will need
the notion of Kleene expressions and Zimin words as discussed in the next section. To this end, we need to break the 
lumping into two steps. First we project each semaphore code word $s \in \mathcal{S}(S,A)$ 
(or path possibly with loops from $\mathbbm{1}$ to $\mathcal{N}(S,A)$) to its normal form (or endpoint) in 
$\mathcal{N}(S,A)$ given by $\tau(s)$. In other words, we map each code word $s \in \mathcal{S}(S,A)$ to the word 
resulting from $s$ by reading it in $\Gamma(S,A)$ with the ``loops stripped away''.
Next we identify any normal forms in $\mathcal{N}(S,A)$ that represent the same element in
$\mathsf{KR}(S,A)$.

For the stationary distribution of the projected Markov chain, we need the preimage of the normal forms
of Definition~\ref{definition.normal}.

\begin{definition}
\label{definition.NF inv}
For a normal form $w \in \mathcal{N}(S,A)$, define
\[
	\mathsf{NF}^{-1}(w) = \{ s \in \mathcal{S}(S,A) \mid \tau(s)=w \} 
	\subseteq \mathcal{S}(S,A).
\]
Furthermore, define
\[
\begin{aligned}
	\mathsf{Red}_{\mathsf{KR}(S,A)}(w) &= \{n \in \mathcal{N}(S,A) \mid [n]_{\mathsf{KR}(S,A)} = w\}
	& &\text{for $w \in K(\mathsf{KR}(S,A))$}.
\end{aligned}
\]
\end{definition}

Note that, if the graph $\Gamma(S,A)$ has non-trivial connected components, then the paths in $\mathsf{NF}^{-1}(w)$
from the root to the ideal can be unbounded in length. With Definition~\ref{definition.NF inv}, we rewrite the stationary
distributions of Corollary~\ref{corollary.stationary} as
\begin{equation}
\label{equation.Psi NF inv}
	\Psi_w^{\mathsf{KR}(S,A)} = \sum_{v \in \mathsf{Red}_{\mathsf{KR}(S,A)}(w)} \;\;
	\sum_{s \in  \mathsf{NF}^{-1}(v)} \Psi_s^{\mathcal{S}(S,A)}.
\end{equation}
Furthermore, if the McCammond expansion of the Karnofsky--Rhodes expansion of $\mathsf{RCay}(S,A)$ is stable,
that is $\mathsf{Mc} \circ \mathsf{KR}(S,A) = \mathsf{KR}(S,A)$, then $\mathsf{Red}_{\mathsf{KR}(S,A)}(w)=\{w\}$ 
and~\eqref{equation.Psi NF inv} simplifies to
\begin{equation}
\label{equation.Psi NF stable}
	\Psi_w^{\mathsf{KR}(S,A)} = \sum_{s \in  \mathsf{NF}^{-1}(w)} \Psi_s^{\mathcal{S}(S,A)}.
\end{equation}

Finally, the Markov chain $\mathcal{M}(S,A)$ can be obtained from $\mathcal{M}(\mathsf{KR}(S,A))$
by lumping as well.

\begin{corollary}
\label{corollary.stationary real}
The Markov chain $\mathcal{M}(S,A)$ is the lumping of $\mathcal{M}(\mathsf{KR}(S,A))$ with 
stationary distribution
\[
	\Psi_w^{(S,A)} = \sum_{\substack{v \in \mathsf{KR}(S,A)\\ [v]_{(S,A)} =w}} \; \Psi_v^{\mathsf{KR}(S,A)}
	\qquad \text{for all $w \in (S,A)$.}
\]
\end{corollary}

\begin{proof}
The lumping condition~\eqref{equation.lumping} follows from the fact that both $(S,A)$ and $\mathsf{KR}(S,A)$
are semigroups.
\end{proof}

\subsection{Kleene expressions and Zimin words}
\label{section.kleene}

We are now going to discuss how to compute the stationary distribution $\Psi_w^{\mathsf{KR}(S,A)}$ of 
Corollary~\ref{corollary.stationary} more explicitly using Kleene expressions and Zimin words.

Given a set $L$, set $L^0 = \{\varepsilon \}$ given by the empty string, $L^1 = L$, and recursively
$L^{i+1} = \{wa \mid w \in L^i, a \in L\}$ for each integer $i>0$. Then the \defn{Kleene star} of $L$ is
\[
	L^\star = \bigcup_{i\geqslant 0} L^i.
\]

The collection of \defn{regular languages} over an alphabet $A$ is the smallest collection of languages
\begin{itemize}
\item containing the empty language and the singletons $\{a\}$ for $a \in A$;
\item closed under union, concatenation and Kleene star.
\end{itemize}
A \defn{Kleene expression} is an expression involving letters in $A$, unions, $\cdot$, and $\star$
and is a compact way to describe a regular language.

\defn{Zimin words} allow us to rewrite the star of a union in terms of just products and star:
\begin{equation}
\label{equation.star union}
	\left( \{a\} \cup \{b\} \right)^\star = \{a,b\}^\star = a^\star ( ba^\star)^\star.
\end{equation}
Iterating~\eqref{equation.star union}, we can also express the star of the union of three elements in terms of multiplication
and star as
\[
	\{a,b,c\}^\star = \{a,b\}^\star (c \{a,b\}^\star)^\star = a^\star (ba^\star)^\star \left( c a^\star (ba^\star)^\star\right)^\star,
\]
or more generally 
\begin{multline}
\label{equation.zimin}
	\{a_1,a_2,\ldots,a_n\}^\star = \{a_1,\ldots, a_{n-1}\}^\star ( a_n \{a_1,\ldots,a_{n-1}\}^\star)^\star\\
	= \{a_1,\ldots,a_{n-2}\}^\star \left( a_{n-1} \{a_1,\ldots,a_{n-2}\}^\star \right)^\star
	  \left( a_n \{a_1,\ldots,a_{n-2}\}^\star \left( a_{n-1} \{a_1,\ldots,a_{n-2}\}^\star \right)^\star \right)^\star = \cdots .
\end{multline}

The following result was proven in~\cite[Section 3.2]{MRS.2011}.

\begin{theorem} \cite{MRS.2011}
For $w \in \mathcal{N}(S,A)$, $\mathsf{NF}^{-1}(w)$ has a Kleene expression without union, that is,
only using letters in $A$, $\cdot$ and $\star$.
\end{theorem}

We can now use a Kleene expression for $\mathsf{NF}^{-1}(w)$ in~\eqref{equation.Psi NF inv} 
to compute the stationary distribution $\Psi_w^{\mathsf{KR}(S,A)}$. The advantage in doing so is
that one can immediately obtain rational expressions. Namely, using the geometric series, we find that
\[
	\sum_{s \in a^\star} \; \Psi^{\mathcal{S}}_s = \sum_{\ell=0}^\infty x_a^\ell = \frac{1}{1-x_a}.
\]
Similarly, using the Zimin words~\eqref{equation.star union}
\[
	\sum_{s \in \{a,b\}^\star} \Psi^{\mathcal{S}}_s = \sum_{s \in a^\star (ba^\star)^\star} \Psi^{\mathcal{S}}_s 
	= \frac{1}{1-x_a} \cdot \frac{1}{1-\frac{x_b}{1-x_a}}
	= \frac{1}{1-x_a-x_b}.
\]
In general, using the recursion~\eqref{equation.zimin} we derive by induction
\begin{equation}
\label{equation.geometric}
	\sum_{s \in \{a_1,a_2,\ldots,a_n\}^\star} \Psi^{\mathcal{S}}_s = \frac{1}{1-x_{a_1} - x_{a_2} - \cdots - x_{a_n}}.
\end{equation}

To take advantage of this, it is important that every element in $\mathsf{NF}^{-1}(v)$
appearing in~\eqref{equation.Psi NF inv} occurs exactly once in the Kleene expression. 
This is not necessarily true for any Kleene expression. For example, in $a^\star a^\star$  the letter $a$ appears
more than once. The condition is, however, ensured if
\begin{equation}
\label{equation.prob condition}
	\sum_{w\in K(\mathsf{KR}(S,A))} \Psi^{\mathsf{KR}(S,A)}_w = 1.
\end{equation}

\begin{conjecture}
\label{conjecture.kleene}
The Kleene expressions constructed in~\cite{MRS.2011} satisfy~\eqref{equation.prob condition}.
\end{conjecture}

In all examples worked out in this paper, Conjecture~\ref{conjecture.kleene} holds.
We believe that the proof is a straightforward induction on the proof in~\cite{MRS.2011} of the existence of such
Kleene expressions, but the details are beyond the goal of this paper.

\begin{example}
Continuing Examples~\ref{example.markov} and~\ref{example.markov lumped}, we find that for the elements
in $\mathcal{N}(S',A)$
\[
	\mathsf{NF}^{-1}(a) = a, \qquad  \mathsf{NF}^{-1}(ba) = b(b^2)^\star a, \qquad
	\mathsf{NF}^{-1}(b^2a) = b^2 (b^2)^\star a.
\]
Since $K(S')$ is left zero, we find by~\eqref{equation.Psi NF stable} that for $w \in \mathcal{N}(S',A)$
\[
	\Psi_w^{\mathsf{KR}(S',A)} = \sum_{s\in \mathsf{NF}^{-1}(w)} \Psi_s^{\mathcal{S}(S',A)}
\]
and specifically
\[
	\Psi_{a}^{\mathsf{KR}(S',A)} = x_a, \qquad \Psi_{ba}^{\mathsf{KR}(S',A)} = \frac{x_a x_b}{1-x_b^2},
	\qquad \Psi_{b^2a}^{\mathsf{KR}(S',A)} = \frac{x_a x^2_b}{1-x_b^2}.
\]
Note that since $x_a+x_b=1$
\[
	\Psi_{a}^{\mathsf{KR}(S',A)} + \Psi_{ba}^{\mathsf{KR}(S',A)} + \Psi_{b^2a}^{\mathsf{KR}(S',A)} =
	x_a + \frac{x_a x_b (1+x_b)}{1-x_b^2} = x_a + \frac{x_a x_b}{1-x_b} = x_a + x_b =1,
\]
verifying~\eqref{equation.prob condition}.
\end{example}

For another example, see Section~\ref{section.Tsetlin library}.

\subsection{The bar and flat operation}
\label{section.bar and flat}

We will now discuss the \defn{bar} and \defn{flat operation}~\cite{LRS.2017}, which will make it possible to extend 
Corollary~\ref{corollary.stationary} to any finite $A$-semigroup $S$, not just those whose minimal ideal $K(S)$ is left zero.

In~\cite{LRS.2017}, $(S,A)^{\mathsf{bar}}$ was defined by considering the right regular representation of $s$
acting faithfully on $S^{\mathbbm{1}}$ and by adding all constant maps on $S^{\mathbbm{1}}$; the constant map
onto $s\in S^{\mathbbm{1}}$ is denoted by $\overline{s}$. The semigroup multiplication of $(S,A)^{\mathsf{bar}}$
is the composition of functions acting on the right. 
In this spirit, let $(S,A)^{\mathsf{bar}}=(S^\mathsf{bar},A\cup \{\overline{\mathbbm{1}}\})$, where $S^\mathsf{bar} = 
S \cup \overline{S} \cup \overline{\{\mathbbm{1}}\}$ and $\overline{S} = \{\overline{x} \mid x \in S\}$. 
The element $\overline{\mathbbm{1}}$ acts as a constant map 
$z \cdot \overline{\mathbbm{1}} = \overline{\mathbbm{1}}$ and in addition $\overline{\mathbbm{1}} \cdot z = \overline{z}$ 
for any $z\in S^\mathsf{bar}$, where we interpret $\overline{\overline{x}}=\overline{x}$ if $z=\overline{x}$. Furthermore, for 
any $x,y\in S$
\[
	x \cdot \overline{y} = \overline{y}, \qquad \overline{x} \cdot \overline{y} = \overline{y}, \qquad \text{and} \qquad
	\overline{x} \cdot y = \overline{x\cdot y}.
\]

Given a semigroup $S$, let $S^\mathsf{op}$ denote the semigroup obtained by reversing the multiplication on $S$.
Then $(S,A)^\flat = (((S,A)^\mathsf{op})^\mathsf{bar})^\mathsf{op}$. The relations with respect to $\mathsf{bar}$ 
get reversed, that is, $\widetilde{\mathbbm{1}} \cdot z = \widetilde{\mathbbm{1}}$ and $z \cdot \widetilde{\mathbbm{1}} = 
\widetilde{z}$ for any $z\in S^\flat$. Furthermore, for any $x,y\in S$
\[
	\widetilde{y} \cdot x = \widetilde{y}, \qquad \widetilde{y} \cdot \widetilde{x} = \widetilde{y}, \qquad \text{and} \qquad
	y \cdot \widetilde{x} = \widetilde{y\cdot x}.
\]
In particular, $K((S,A)^\flat) = \{\widetilde{z} \mid z \in S\} \cup \{\widetilde{\mathbbm{1}}\}$, which is left zero.

\begin{remark}
Note that up to the labeling of the vertices, $\mathsf{KR}(S\cup \{\square\}, A\cup \{\square\})$ is exactly 
$\left(\mathsf{KR}(S,A)\right)^\flat$. This is true since right multiplication in $\left(\mathsf{KR}(S,A)\right)^\flat$
by any $\widetilde{z}$ with $z\in S$ or $\widetilde{\mathbbm{1}}$ lands in $K(\left(\mathsf{KR}(S,A)\right)^\flat)$ 
and similarly right multiplication by $\square$ in $\mathsf{KR}(S\cup \{\square\}, A\cup \{\square\})$ also lands
in $K(\mathsf{KR}(S\cup\{\square\},A\cup\{\square\}))$.
Hence the flat $\flat$ operation can be interpreted as adding a new zero to the semigroup.
\end{remark}

We may now generalize Corollary~\ref{corollary.stationary}.
Recall the hypothesis of Corollary~\ref{corollary.stationary} that $K(S)$ is left zero. In 
Corollary~\ref{corollary.stationary general} below there is no restriction on $K(S)$.
The proof uses the flat construction to reduce the general case to the previous case and then limiting the 
probability of the new variable to zero.

\begin{corollary}
\label{corollary.stationary general}
If $K(S)$ is not left zero, the Markov chain $\mathcal{M}(\mathsf{KR}(S,A))$ has stationary distribution
\begin{equation}
\label{equation.psi limit}
	\Psi^{\mathsf{KR}(S,A)}_w 
	= \lim_{x_\square \to 0} \bigl(\sum_{\substack{s \in \mathcal{S}(S\cup \{\square\},A\cup\{\square\})\\ 
	[s]_{\mathsf{KR}(S\cup \{\square\},A\cup \{\square\})}=w\square}}\; \prod_{a\in s} x_a \bigr)
	\qquad \text{for all $w \in K(\mathsf{KR}(S,A))$.}
\end{equation}
\end{corollary}

\begin{proof}
Consider $\mathsf{KR}(S', A')$ where $S'=S\cup \{\square\}$ and $A'=A\cup \{\square\}$ (or equivalently 
$\left(\mathsf{KR}(S,A)\right)^\flat$). The minimal ideal of this semigroup consists of all elements $w\square$, where 
$w\in S$. Since $\square v = \square$ for all $v \in S$, the minimal ideal is left zero. Hence 
Corollary~\ref{corollary.stationary} applies to $\mathcal{M}(\mathsf{KR}(S', A'))$ and
\begin{equation*}
	\Psi^{\mathsf{KR}(S',A')}_{w\square} = 
	\sum_{\substack{s \in \mathcal{S}(S',A')\\ 
	[s]_{\mathsf{KR}(S',A')}=w\square}} \Psi^{\mathcal{S}(S',A')}_s
	\qquad \text{for all $w\in S$.}
\end{equation*}
Taking the limit $x_\square$ to zero, the states with nonzero stationary probability will be in
$K(\mathsf{KR}(S,A))$ and we obtain~\eqref{equation.psi limit} for the stationary distribution.
\end{proof}

\begin{example}
Recall the semigroup $S=Z_2 \times \{0,1\}$ with generators $A=\{a,b\}$ with $a=(z,0)$ and $b=(z,1)$, where
$z$ is the generator of $Z_2$ with $z^2=1$, from Example~\ref{example.markov}.
If we add a generator $\square$ which acts as zero, then the normal forms of $\Gamma(S',A') = 
\mathsf{Mc} \circ \mathsf{KR}(S',A')$ with $S'=S \cup \{\square\}$ and $A'=A \cup \{\square\}$ are given by
\[
	\mathcal{N}(S',A') = \{\square, a \square,a^2 \square ,ab \square,b \square,
	b^2 \square, ba \square, ba^2 \square, bab \square, b^2a \square, b^2 a^2 \square, b^2 ab \square\}.
\]
We have 
\[
\begin{aligned}
	\mathsf{NF}^{-1}(\square) &= \square, & 
	\mathsf{NF}^{-1}(a\square) &= a \left( \{a,b\}^2 \right)^\star \square,\\
	\mathsf{NF}^{-1}(a^2 \square) &= a \left( \{a,b\}^2 \right)^\star a \square, &
	\mathsf{NF}^{-1}(ab \square) &= a \left( \{a,b\}^2 \right)^\star b \square,\\
	\mathsf{NF}^{-1}(b \square) &= b (b^2)^\star \square, &
	\mathsf{NF}^{-1}(b^2 \square) &= b (b^2)^\star b \square,\\
	\mathsf{NF}^{-1}(ba\square) &= b (b^2)^\star a \left( \{a,b\}^2 \right)^\star \square, &
	\mathsf{NF}^{-1}(ba^2\square) &= b (b^2)^\star a \left( \{a,b\}^2 \right)^\star a \square,\\
	\mathsf{NF}^{-1}(bab\square) &= b (b^2)^\star a \left( \{a,b\}^2 \right)^\star b \square, &
	\mathsf{NF}^{-1}(b^2 a\square) &= b (b^2)^\star ba \left( \{a,b\}^2 \right)^\star \square,\\
	\mathsf{NF}^{-1}(b^2 a^2\square) &= b (b^2)^\star ba \left( \{a,b\}^2 \right)^\star a \square, &
	\mathsf{NF}^{-1}(b^2 ab\square) &= b (b^2)^\star ba \left( \{a,b\}^2 \right)^\star b \square.
\end{aligned}
\]
Hence the stationary distribution for $\mathcal{M}(\mathsf{KR}(S',A'))$ is
\[
\begin{aligned}
	\Psi^{\mathsf{KR}(S',A')}_\square &= x_\square, & 
	\Psi^{\mathsf{KR}(S',A')}_{a\square} &= \frac{x_a x_\square}{1-(x_a+ x_b)^2},\\
	\Psi^{\mathsf{KR}(S',A')}_{a^2 \square} &= \frac{x_a (x_a+x_b) x_\square}{1-(x_a+x_b)^2}, &
	\Psi^{\mathsf{KR}(S',A')}_{b\square} &= \frac{x_b x_\square}{1-x_b^2},\\
	\Psi^{\mathsf{KR}(S',A')}_{b^2 \square} &= \frac{x_b^2 x_\square}{1-x_b^2}, &
	\Psi^{\mathsf{KR}(S',A')}_{ba\square} &= \frac{x_a x_b x_\square}{(1-x_b^2)(1-(x_a+x_b)^2)},\\
	\Psi^{\mathsf{KR}(S',A')}_{ba^2\square} &= \frac{x_a (x_a+x_b) x_b x_\square}{(1-x_b^2)(1-(x_a+x_b)^2)}, &
	\Psi^{\mathsf{KR}(S',A')}_{b^2 a\square} &= \frac{x_a x^2_b x_\square}{(1-x_b^2)(1-(x_a+x_b)^2)},\\
	\Psi^{\mathsf{KR}(S',A')}_{b^2 a^2\square} &= \frac{x_a (x_a+x_b) x^2_b x_\square}{(1-x_b^2)(1-(x_a+x_b)^2)}. &&
\end{aligned}
\]
Note that $\Psi^{\mathsf{KR}(S',A')}_{a\square}+\Psi^{\mathsf{KR}(S',A')}_{a^2\square} = x_a$, 
$\Psi^{\mathsf{KR}(S',A')}_{b \square} + \Psi^{\mathsf{KR}(S',A')}_{b^2\square} = \frac{x_\square x_b} {1-x_b}$, and
\[
	\Psi^{\mathsf{KR}(S',A')}_{ba\square} + \Psi^{\mathsf{KR}(S',A')}_{ba^2 \square} 
	+ \Psi^{\mathsf{KR}(S',A')}_{b^2a\square} + \Psi^{\mathsf{KR}(S',A')}_{b^2 a^2 \square} 
	= \frac{x_a x_b x_\square}{(1-x_b)(1-x_a-x_b)}.
\]
Hence the total sum of all stationary states is $x_\square + x_a + x_b =1$.

Taking the limit $x_\square \to 0$, we obtain the stationary distribution for $\mathcal{M}(\mathsf{KR}(S,A))$.
Note that using $x_\square + x_a + x_b =1$, we have $1-(x_a+x_b)^2 = 1-(1-x_\square)^2 = 2x_\square - x_\square^2$.
Hence
\[
\begin{aligned}
	\Psi^{\mathsf{KR}(S,A)}_{a} &= \frac{x_a}{2},& \qquad
	\Psi^{\mathsf{KR}(S,A)}_{a^2} &= \frac{x_a}{2},\\
	\Psi^{\mathsf{KR}(S,A)}_{b} &= 0,& \qquad
	\Psi^{\mathsf{KR}(S,A)}_{b^2} &= 0,\\
	\Psi^{\mathsf{KR}(S,A)}_{ba} &= \frac{x_a x_b}{2(1-x_b^2)},& \qquad
	\Psi^{\mathsf{KR}(S,A)}_{ba^2} &= \frac{x_a x_b}{2(1-x_b^2)},\\
	\Psi^{\mathsf{KR}(S,A)}_{b^2 a} &= \frac{x_a x^2_b}{2(1-x_b^2)},& \qquad
	\Psi^{\mathsf{KR}(S,A)}_{b^2a^2} &= \frac{x_a x^2_b}{2(1-x_b^2)}.
\end{aligned}
\]

We can lump further to $\mathcal{M}(S,A)$ by using that $a=b^2a=ba^2$ and
$a^2=ba=b^2a^2$ in $S$, so that
\[
	\Psi^{(S,A)}_a = \Psi_{a^2}^{(S,A)} = \frac{x_a}{2}\left(1+\frac{x_b^2}{1-x_b^2} + \frac{x_b}{1-x_b^2}\right)
	= \frac{x_a}{2}\left(1+\frac{x_b}{1-x_b}\right) = \frac{x_a}{2} + \frac{x_b}{2} = \frac{1}{2}.
\] 
\end{example}

\subsection{Bounds on the mixing time}
\label{section.mixing time}

In this section we use the techniques developed in~\cite{ASST.2015} combined with the
normal forms coming from the McCammond expansion to provide an upper bound on the mixing time
for $\mathcal{M}(\mathsf{KR}(S,A))$ and $\mathcal{M}(S,A)$.

The total variation distance between two probability distributions $\nu$ and $\mu$ is defined by
\[
	\|\nu - \mu \|_{TV} = \max_{B \subseteq S} | \nu(B) - \mu(B)|,
\]
where $\nu(B) = \sum_{s \in B} \nu(s)$.
Let $\mathcal M$ be a finite state irreducible Markov chain with stationary distribution $\Psi$.
Let $d(k) = \sup_\nu \| \mathcal{T}^k \nu - \Psi \|$. Then, for $\varepsilon > 0$, the \defn{mixing time} of 
$\mathcal{M}$ is~\cite{LPW.2009}
\[
	t_{\mathrm{mix}}(\varepsilon) = \min \{k \mid d(k) \leqslant \varepsilon\}.
\]
Often authors choose $\varepsilon = e^{-1}$ or $\varepsilon = 1/4$ to define the mixing time. We bound
for $c > 0$, when $\| \mathcal{T}^k \nu - \Psi\| \leqslant e^{-c}$.

\begin{lemma} \cite[Lemma 3.6]{ASST.2015}
\label{lemma.statisticbound}
Let $\mathcal M$ be an irreducible Markov chain associated to the semigroup $S$ and probability distribution
$0\leqslant p(s) \leqslant 1$ for $s\in S$. We assume that $\{s \in S \mid p(s)>0\}$ generates $S$.
Let $\Psi$ be the stationary distribution and $f\colon S\to \mathbb N$ be a function, 
called a \defn{statistic}, such that:
\begin{enumerate}
\item $f(s's)\leqslant f(s)$ for all $s,s'\in S$;
\item if $f(s)>0$, then there exists $s' \in S$ with $p(s')>0$ such that $f(s's)<f(s)$;
\item $f(s)=0$ if and only if $s \in K(S)$.
\end{enumerate}
Then if $p=\min\{p(s) \mid s \in S, p(s)>0\}$ and $n=f(\mathbbm{1})$, we have that
\[
  \|\mathcal T^k\nu -\Psi\|_{TV} \leqslant \sum_{i=0}^{n-1} {k\choose i}p^i(1-p)^{k-i}
  \leqslant \exp\left(-\frac{(kp-(n-1))^2}{2kp}\right)\,,
  \]
for any probability distribution $\nu$ on $S$, where the last inequality holds as long as $k\geqslant (n-1)/p$.
\end{lemma}

Now consider $\mathcal{M}(S,A)$ for an $A$-semigroup $(S,A)$ with probabilities $0<x_a \leqslant 1$ on the 
generators. Let $n$ be the maximal length of a chain from $\mathbbm{1}$ to a leaf in the
McCammond expansion $\mathsf{Mc} \circ \mathsf{KR}(S,A)$. Let $\ell$ be the maximal distance between two 
transition edges in $\mathsf{Mc} \circ \mathsf{KR}(S,A)$ (which is well-defined since $\mathsf{Mc} \circ \mathsf{KR}(S,A)$
is a tree). Define the statistic $f(s)$ to be the maximum of the number of transition 
arrows in the unique path from $v$ to a leaf in $\mathsf{Mc} \circ \mathsf{KR}(S,A)$, where $v$ runs over all vertices
in $\mathsf{Mc} \circ \mathsf{KR}(S,A)$ such that $[v]_{\mathsf{KR}(S,A)}=s$. Note that this statistic is constant
on $\mathscr{R}$-classes. Furthermore, it satisfies Conditions~(1) and~(3) of Lemma~\ref{lemma.statisticbound}.
Condition~(2) is not necessarily satisfied. However, if we take $\ell$ steps at a time in the original Markov
chain $\mathcal{M}(S,A)$, then with probability at least $p^\ell>0$, where $p = \min\{x_a \mid a\in A\}$, the statistic
strictly decreases. Hence, in this random walk, where we take $\ell$ steps at a time as compared to $\mathcal{M}(S,A)$, 
Lemma~\ref{lemma.statisticbound} applies and we obtain
\[
  \|\mathcal T^k\nu -\Psi\|_{TV}  \leqslant \exp\left(-\frac{(kp^\ell/\ell-(n-1)/\ell)^2}{2kp^\ell/\ell}\right)
  = \exp\left(-\frac{(kp^\ell-(n-1))^2}{2k\ell p^\ell}\right)\;.
\]
This shows that the mixing time is at most $2(n+\ell c-1)/p^\ell$ (see also~\cite[Section 6]{ASS.2014}).

\section{Examples}
\label{section.examples}

In this section we derive explicit Markov chains from the general theory of Section~\ref{section.walks} for
specific choices of $A$-semigroups $(S,A)$. We begin by treating known examples, such as
the Tsetlin library, in this setting and then move on to new examples.

\subsection{The Tsetlin library}
\label{section.Tsetlin library}

The Tsetlin library~\cite{Tsetlin.1963} is a Markov chain whose states are all permutations $S_n$ of $n$ books (on a shelf).
Given $\pi \in S_n$, construct $\pi' \in S_n$ from $\pi$ by removing book $a$ from the shelf and inserting 
it to the front. In this case write $\pi \stackrel{a}{\longrightarrow} \pi'$. 
Let $0< x_a\leqslant 1$ be probabilities for each $1\leqslant a \leqslant n$ such that $\sum_{a=1}^n x_a = 1$. 
In the Tsetlin library Markov chain, we transition $\pi \stackrel{a}{\longrightarrow} \pi'$ with probability $x_a$. 
The stationary distribution for the Tsetlin library was derived by Hendricks~\cite{Hendricks.1972, Hendricks.1973}
\begin{equation}
\label{equation.psi Tsetlin}
	\Psi_\pi = \prod_{i=1}^{n} \frac{x_{\pi_{i}}}{x_{\pi_{i+1}}+ \cdots + x_{\pi_n}} \qquad \text{for all $\pi \in S_n$.}
\end{equation}
We are now going to derive the stationary distribution using the methods developed in Section~\ref{section.walks}.

Consider the semigroup $P(n)$, which consists of the set of all non-empty subsets of $\{1,2,\ldots,n\}$. Multiplication in 
$P(n)$ is union of sets. We pick as generators $A=[n]:=\{1,2,\ldots,n\}$. Then the right Cayley graph $\mathsf{RCay}(P(n),[n])$ 
is the Boolean poset with $\mathbbm{1}$ as root. The right Cayley graph for $P(3)$ is depicted in Figure~\ref{figure.P3}.
Except for the loops at a given vertex, all edges are transitional. Hence $\Gamma(P(n),[n])= 
\mathsf{Mc} \circ \mathsf{KR}(P(n),[n]) = \mathsf{KR}(P(n),[n])$ is a tree with leaves given by the permutations $S_n$ of $[n]$.
The case $n=3$ is given in Figure~\ref{figure.Mc P3}. The minimal ideal is $K(P(n)) = [n]$ and 
\[
	\mathcal{N}(P(n),[n]) = S_n.
\]
Let $\pi \in S_n$ be a permutation. Then $\mathsf{NF}^{-1}(\pi)$ consists of all paths in $\Gamma(P(n),[n])$
starting at $\mathbbm{1}$ and ending at $\pi=\pi_1\pi_2 \ldots \pi_n$. Such a path has to pass through the vertex
$\pi_1 \ldots \pi_i$ for each $1\leqslant i\leqslant n$. At a given vertex $\pi_1 \ldots \pi_i$, it can loop
with $\pi_1,\ldots,\pi_i$. Hence we can write the paths to $\pi$ via Kleene expressions as
\[
	\mathsf{NF}^{-1}(\pi) = \pi_1 \pi_1^\star \pi_2 \{\pi_1,\pi_2\}^\star \cdots \pi_i \{\pi_1,\ldots,\pi_i\}^\star \cdots \pi_n.
\]
Using the Zimin words~\eqref{equation.zimin}, this expression can be written entirely in terms of star and
multiplication without the sets.

The states of the Markov chain $\mathcal{M}^{\mathcal{S}(P(n),[n])}$ consist of all words in the alphabet $[n]$ that end
once all letters in $[n]$ are used. The state space of the lumped Markov chain $\mathcal{M}(\mathsf{KR}(P(n),[n])$ 
is $S_n$. We transition $\pi \stackrel{a}{\longrightarrow} \pi'$ with probability $x_a$, where $\pi'$ is obtained from $\pi$ by 
prepending $a$ to $\pi$ and removing the letter $a$ from $\pi$. Equivalently, this corresponds to moving the letter $a$ in 
$\pi$ to the front, which is exactly the transition in the Tsetlin library.

By~\eqref{equation.Psi NF stable}, the stationary distribution associated to $\pi \in S_n$ is
\[
	\Psi^{\mathsf{KR}(P(n),[n])}_\pi = \sum_{s \in \mathsf{NF}^{-1}(\pi)} \; \prod_{a\in s} x_a.
\]
Using~\eqref{equation.geometric}, this can be rewritten as
\[
	\Psi^{\mathsf{KR}(P(n),[n])}_\pi = \prod_{i=1}^n \frac{x_{\pi_i}}{1-\sum_{j=1}^{i-1} x_{\pi_j}}
	= \prod_{i=1}^n \frac{x_{\pi_i}}{\sum_{j=i+1}^n x_{\pi_j}},
\]
where we used that $\sum_{j=1}^n x_{\pi_j} = 1$. This agrees with~\eqref{equation.psi Tsetlin}.

\begin{figure}[t]
\begin{tikzpicture}[auto]
\node (I) at (0, 0) {$\mathbbm{1}$};
\node (A) at (-3,-1.5) {$\{1\}$};
\node(B) at (0,-1.5) {$\{2\}$};
\node(C) at (3,-1.5) {$\{3\}$};
\node(D) at (-3,-3) {$\{1,2\}$};
\node(E) at (0,-3) {$\{1,3\}$};
\node(F) at (3,-3) {$\{2,3\}$};
\node(G) at (0,-4.5) {$\{1,2,3\}$};

\draw[edge,blue,thick] (I) -- (A) node[midway, left] {$1$\;};
\draw[edge,blue,thick] (I) -- (B) node[midway, right] {$2$};
\draw[edge,blue,thick] (I) -- (C) node[midway, right] {\;$3$};
\draw[edge,blue,thick] (A) -- (D) node[midway,left] {$2$};
\draw[edge,blue,thick] (A) -- (E) node[midway,above] {$3$};
\draw[edge,blue,thick] (B) -- (D) node[midway,below] {$1$};
\draw[edge,blue,thick] (C) -- (F) node[midway,right] {$2$};
\draw[edge,blue,thick] (B) -- (F) node[midway,below] {$3$};
\draw[edge,blue,thick] (C) -- (E) node[midway,above] {$1$};
\draw[edge,blue,thick] (D) -- (G) node[midway,below] {$3$};
\draw[edge,blue,thick] (E) -- (G) node[midway,right] {$2$};
\draw[edge,blue,thick] (F) -- (G) node[midway,right] {\;$1$};
\path
	(A) edge [loop left] node {$1$} (A)
	(B) edge [loop right] node {$2$} (B)
	(C) edge [loop right] node {$3$} (C)
	(D) edge [loop left] node {$1,2$} (D)
	(E) edge [loop right] node {$1,3$} (E)
	(F) edge [loop right] node {$2,3$} (F)
	(G) edge [loop right] node {$1,2,3$} (G);
\end{tikzpicture}
\caption{The right Cayley graph $\mathsf{RCay}(S,A)$ with $S=P(3)$ and $A=\{1,2,3\}$. Transition edges are drawn in blue.
\label{figure.P3}}
\end{figure}
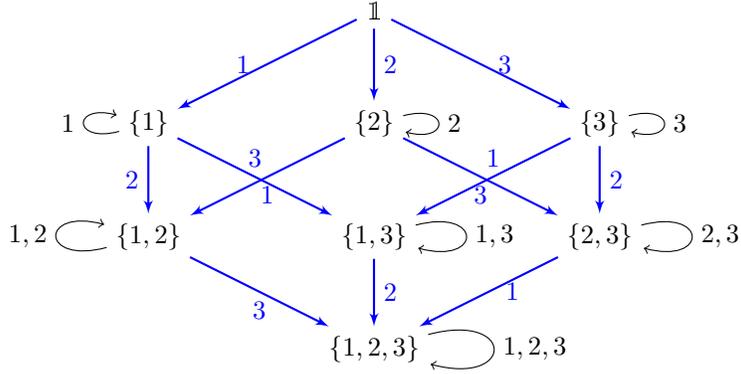

\begin{figure}[t]
\begin{tikzpicture}[auto]
\node (I) at (0, 0) {$\mathbbm{1}$};
\node (A) at (-4,-1.5) {$1$};
\node(B) at (0,-1.5) {$2$};
\node(C) at (4,-1.5) {$3$};
\node(D) at (-5,-3) {$12$};
\node(E) at (-3,-3) {$13$};
\node(F) at (-1,-3) {$21$};
\node(G) at (1,-3) {$23$};
\node(H) at (3,-3) {$31$};
\node(J) at (5,-3) {$32$};
\node(DD) at (-5,-4.5) {$123$};
\node(EE) at (-3,-4.5) {$132$};
\node(FF) at (-1,-4.5) {$213$};
\node(GG) at (1,-4.5) {$231$};
\node(HH) at (3,-4.5) {$312$};
\node(JJ) at (5,-4.5) {$321$};

\draw[edge,blue,thick] (I) -- (A) node[midway, left] {$1$\;\;};
\draw[edge,blue,thick] (I) -- (B) node[midway, right] {$2$};
\draw[edge,blue,thick] (I) -- (C) node[midway, right] {\;\;$3$};
\draw[edge,blue,thick] (A) -- (D) node[midway, left] {$2$};
\draw[edge,blue,thick] (A) -- (E) node[midway, right] {$3$};
\draw[edge,blue,thick] (B) -- (F) node[midway, left] {$1$};
\draw[edge,blue,thick] (B) -- (G) node[midway, right] {$3$};
\draw[edge,blue,thick] (C) -- (H) node[midway, left] {$1$};
\draw[edge,blue,thick] (C) -- (J) node[midway, right] {$3$};
\draw[edge,blue,thick] (D) -- (DD) node[midway, left] {$3$};
\draw[edge,blue,thick] (E) -- (EE) node[midway, left] {$2$};
\draw[edge,blue,thick] (F) -- (FF) node[midway, left] {$3$};
\draw[edge,blue,thick] (G) -- (GG) node[midway, right] {$1$};
\draw[edge,blue,thick] (H) -- (HH) node[midway, right] {$2$};
\draw[edge,blue,thick] (J) -- (JJ) node[midway, right] {$1$};

\path
	(A) edge [loop left, red, dashed,thick] node {$1$} (A)
	(B) edge [loop right, red, dashed, thick] node {$2$} (B)
	(C) edge [loop right, red, dashed,thick] node {$3$} (C)
	(D) edge [loop left, red, dashed,thick] node {$1,2$} (D)
	(E) edge [loop left, red, dashed,thick] node {$1,3$} (E)
	(F) edge [loop left, red, dashed,thick] node {$1,2$} (F)
	(G) edge [loop right, red, dashed,thick] node {$2,3$} (G)
	(H) edge [loop right, red, dashed,thick] node {$1,3$} (H)
	(J) edge [loop right, red, dashed,thick] node {$2,3$} (J)
	(DD) edge [loop below, red, dashed,thick] node {$1,2,3$} (DD)
	(EE) edge [loop below, red, dashed,thick] node {$1,2,3$} (EE)
	(FF) edge [loop below, red, dashed,thick] node {$1,2,3$} (FF)
	(GG) edge [loop below, red, dashed,thick] node {$1,2,3$} (GG)
	(HH) edge [loop below, red, dashed,thick] node {$1,2,3$} (HH)
	(JJ) edge [loop below, red, dashed,thick] node {$1,2,3$} (JJ);
\end{tikzpicture}
\caption{$\Gamma(P(3),[3])=\mathsf{Mc} \circ \mathsf{KR}(P(3),[3]) = \mathsf{KR}(P(3),[3])$, which is the 
Karnofsky--Rhodes expansion of the right Cayley graph of Figure~\ref{figure.P3}.
\label{figure.Mc P3}}
\end{figure}
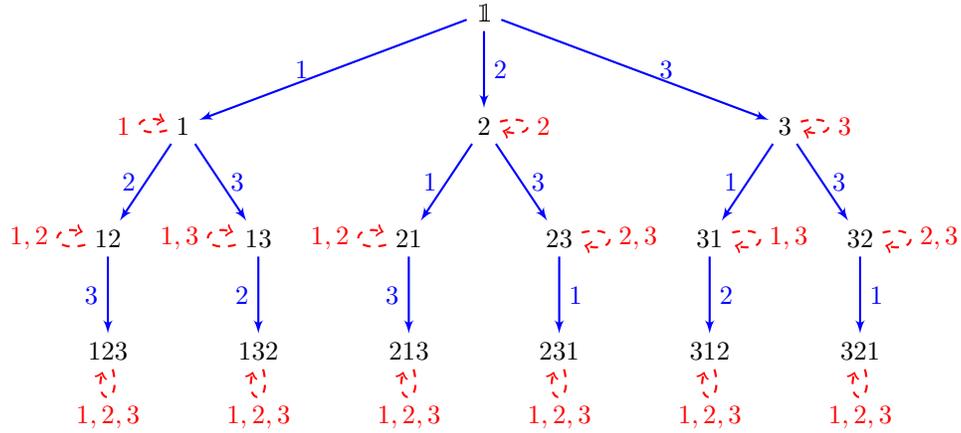

\subsection{Edge flipping on a line}
\label{section.edge flipping}

The example of this section is a Markov chain obtained by edge flipping on a line and was suggested to us 
by Persi Diaconis. It is a Boolean arrangement~\cite{BHR.1999} for which stationary distributions were derived
in~\cite{BD.1998} and which was also analyzed in~\cite{ChungGraham.2012}. Take a line with $n+1$ vertices.
Each vertex can either be $0$ or $1$. So the state space is $\mathcal{S}=\{0,1\}^{n+1}$ of size $2^{n+1}$. 
Pick edge $i$ for $1\leqslant i \leqslant n$ (between vertices $i$ and $i+1$) with probability $x_i$. Then with 
probability $\frac{1}{2}$ make the adjacent vertices both 0 (respectively both 1). Let us call this Markov chain
$\mathcal{M}$.

In our setting, this Markov chain can be treated in a similar fashion to the Tsetlin library.
Let $P^\pm(n)$ be the set of signed subsets of $[n]$, that is, take a subset of $[n]$
and in addition associate to each letter a sign $+$ or $-$. Right multiplication of such a subset 
$X$ by a generator $x \in [\pm n]:=\{\pm 1,\ldots,\pm n\}$ is addition of $x$ to $X$ if neither $x$ nor $-x$ are in $X$
and otherwise return $X$. The minimal ideal in the Karnofsky--Rhodes expansion of this monoid is
the set of signed permutations $S_n^\pm$. Here signed permutations are represented in one-line notation
$\pi_1 \pi_2 \ldots \pi_n$, where $|\pi_1| \ldots |\pi_n|$ is a permutation and $\pi_i \in [\pm n]$
for each $1\leqslant i\leqslant n$.

The Kleene expression for $\pi \in S_n^\pm$ is very similar to the case of the Tsetlin library
\[
	\mathsf{NF}^{-1}(\pi) = \pi_1 \{\pm \pi_1\}^\star \pi_2 \{\pm \pi_1,\pm \pi_2\}^\star \cdots \pi_i 
	\{\pm \pi_1,\ldots,\pm \pi_i\}^\star \cdots \pi_n.
\]
The state space of the lumped Markov chain $\mathcal{M}(\mathsf{KR}(P^\pm(n),[\pm n]))$ 
is $S^\pm_n$. We transition $\pi \stackrel{a}{\longrightarrow} \pi'$ with probability $y_a$ for $a\in [\pm n]$, 
where $\pi'$ is obtained from $\pi$ by prepending $a$ to $\pi$ and removing the letter $a$ or $-a$ from $\pi$. 

By~\eqref{equation.Psi NF stable}, the stationary distribution associated to $\pi \in S^\pm_n$ is
\[
	\Psi^{\mathsf{KR}(P^\pm(n),[\pm n])}_\pi = \sum_{s \in \mathsf{NF}^{-1}(\pi)} \; \prod_{a\in s} y_a.
\]
Using~\eqref{equation.geometric}, this can be rewritten as
\[
	\Psi^{\mathsf{KR}(P^\pm(n),[\pm n])}_\pi = \prod_{i=1}^n \frac{y_{\pi_i}}{1-\sum_{j=1}^{i-1} (y_{\pi_j} + y_{-\pi_j})}.
\]

The edge flipping Markov chain $\mathcal{M}$ can be obtained from $\mathcal{M}(\mathsf{KR}(P^\pm(n),[\pm n]))$
via the action of $P^\pm(n)$ on $\mathcal{S}=\{0,1\}^{n+1}$. For $s\in \mathcal{S}$, the letter $a \in [n]$ acts
on $s=s_1\ldots s_{n+1}$ by changing $s_a$ and $s_{a+1}$ to 0 and $-a$ acts by changing $s_a$ and $s_{a+1}$ to 1.
A signed permutation $\pi \in S_n^\pm$ can be associated with a state $s$ since $s:=\pi.s'$ is independent of 
$s'\in \mathcal{S}$ (since every letter appears once in $\pi$).
Hence setting $y_a=y_{-a}=\frac{x_a}{2}$, we obtain the stationary distribution for $s$ in $\mathcal{M}$
by lumping
\begin{equation}
\label{equation.flipping stationary}
	\Psi_s^\mathcal{M} = \sum_{\substack{\pi \in S_n^\pm\\ s=\pi.0^{n+1}}} \Psi^{\mathsf{KR}(P^\pm(n),[\pm n])}_\pi
	= \frac{1}{2^n}\sum_{\substack{\pi \in S_n^\pm\\ s=\pi.0^{n+1}}} \prod_{i=1}^n \frac{x_i}{1-\sum_{j=1}^{i-1} x_{|\pi_j|}}.
\end{equation}
Note that more generally one could set $y_a = px_a$ and $y_{-a}=(1-p)x_a$ for $0<p<1$. The case above is $p=\frac{1}{2}$.
Formula~\eqref{equation.flipping stationary} has a similar structure as the stationary distributions 
in~\cite[Theorem 2]{BD.1998} and~\cite[Eq. (2.5)]{Denham.2012}.

\begin{example}
For $n=2$, we have
\begin{equation*}
\begin{split}
	\Psi_{12} &= \Psi_{-12} = \Psi_{1-2} = \Psi_{-1-2} = \frac{x_1 x_2}{4(1-x_1)} = \frac{x_1}{4},\\
	\Psi_{21} &= \Psi_{-21} = \Psi_{2-1} = \Psi_{-2-1} = \frac{x_1 x_2}{4(1-x_2)} = \frac{x_2}{4},
\end{split}
\end{equation*}
where we dropped the superscripts. Hence
\begin{equation*}
\begin{split}
	\Psi_{000} &= \Psi_{111} = \Psi_{12} + \Psi_{21} = \frac{x_1+x_2}{4} = \frac{1}{4},\\
	\Psi_{001} &= \Psi_{110} = \Psi_{1-2} = \frac{x_1}{4},\\
	\Psi_{010} &= \Psi_{101} = 0,\\
	\Psi_{011} &= \Psi_{100} = \Psi_{-21} = \frac{x_2}{4}.
\end{split}
\end{equation*}
\end{example}

\begin{example}
For $n=3$, we have for example
\begin{equation*}
\begin{split}
	\Psi_{0010} &= \Psi_{1-23} = \frac{x_1 x_2}{8(x_2+x_3)},\\
	\Psi_{0001} &= \Psi_{12-3} + \Psi_{21-3} + \Psi_{2-31}
	= \frac{x_1 x_2 x_3}{8}\left( \frac{1}{x_3(x_2+x_3)} + \frac{1}{x_3(x_1+x_3)} + \frac{1}{x_1(x_1+x_3)} \right).
\end{split}
\end{equation*}
The most likely states are $0000$ and $1111$ with
\[
	\Psi_{0000} = \Psi_{1111} = \sum_{\pi \in S_3} \Psi_\pi + \Psi_{13-2} +\Psi_{31-2}=
	\frac{1}{8} \left( 1 + \frac{x_1 x_3}{1-x_1} + \frac{x_1 x_3}{1-x_3} \right).
\]
\end{example}

In general, the most likely states are $0^{n+1}$ and $1^{n+1}$ since the largest number of summands
contribute in~\eqref{equation.flipping stationary} (see also the important paper~\cite[Section 9]{ChungGraham.2012}).
In particular, all permutations in $S_n$ contribute and since these terms are exactly the stationary distributions of the Tsetlin
library, they sum to one. Hence $\Psi_{0^{n+1}}^{\mathcal{M}} \sim \frac{1}{2^n}$ plus lower order terms in the limit
$n\to \infty$. The states $0101\ldots$ and $1010\ldots$ appear with probability zero and are hence the least likely.
One of the second least likely states is $001010\ldots$. For this state, only one summand 
in~\eqref{equation.flipping stationary} contributes and $\Psi_{001010\ldots}^{\mathcal{M}} \sim \frac{1}{2^n n!}$.

\subsection{Cyclic walks -- Rees matrix semigroup $B(n)$}
\label{section.Bn}

The Rees matrix semigroup $B(n)$ consists of the elements $\{ij \mid 1\leqslant i,j \leqslant n\} \cup \{\square\}$ 
with multiplication
\[
	ij \cdot k\ell = \begin{cases}
	i\ell & \text{if $j=k$,}\\
	\square & \text{otherwise,}
	\end{cases}
\]
and $\square$ acts as zero. Let us choose as generators $A = \{a_i \mid 1\leqslant i \leqslant n\}$, where $a_i = i(i+1)$ 
for  $1\leqslant i < n$ and $a_n = n1$. 

\begin{example}
\label{example.B2}
Let us consider the special case of $S=B(2)$ with generators $A=\{a,b\}$, where $a=12$ and $b=21$. The right Cayley graph 
and its Karnofsky--Rhodes/McCammond expansion $\mathsf{Mc} \circ \mathsf{KR}(S,A)$ are given in Figure~\ref{figure.B2}. 
Note that in this example $\mathsf{KR}(S,A)$ is stable under the McCammond expansion. The minimal ideal
$K(S)=\{\square\}$ and the normal forms are given by
\[
	\mathcal{N}(B(2),A) = \{aa,abb,baa,bb\}.
\]
The Markov chain $\mathcal{M}(\mathsf{KR}(B(2),A))$ is depicted in Figure~\ref{figure.markov chain B2}.
To compute the stationary distribution, we first obtain
\begin{equation*}
\begin{split}
	\mathsf{NF}^{-1}(aa) = a (ba)^\star a, \qquad \qquad &\mathsf{NF}^{-1}(abb) = ab (ab)^\star b,\\
	\mathsf{NF}^{-1}(bb) = b (ab)^\star b, \qquad \qquad & \mathsf{NF}^{-1}(baa) = ba (ba)^\star a.
\end{split}
\end{equation*}
By~\eqref{equation.Psi NF stable}, we obtain the stationary distribution
\begin{equation*}
\begin{split}
	\Psi^{\mathsf{KR}(S,A)}_{aa} = \frac{x_a^2}{1-x_a x_b}, \qquad \qquad 
	& \Psi^{\mathsf{KR}(S,A)}_{abb} = \frac{x_a x_b^2}{1-x_a x_b},\\
	\Psi^{\mathsf{KR}(S,A)}_{bb} = \frac{x_b^2}{1-x_a x_b}, \qquad \qquad 
	& \Psi^{\mathsf{KR}(S,A)}_{baa} = \frac{x_a^2 x_b}{1-x_a x_b}.
\end{split}
\end{equation*}
We indeed verify, using $x_a+x_b=1$, that
\begin{multline*}
	\Psi^{\mathsf{KR}(S,A)}_{aa}+\Psi^{\mathsf{KR}(S,A)}_{bb}+\Psi^{\mathsf{KR}(S,A)}_{abb}
	+\Psi^{\mathsf{KR}(S,A)}_{baa} 
	= \frac{1}{1-x_a x_b}(x_a^2 + x_b^2 + x_a x_b^2 + x_a^2 x_b)\\
	= \frac{1}{1-x_a x_b}( x_a^2 + x_b^2 + x_a x_b)
	= \frac{1}{1-x_a x_b}((x_a+x_b)^2 - x_a x_b) = 1.
\end{multline*}
\end{example}

\begin{figure}[t]
\begin{center}
\begin{tikzpicture}
\node (A) at (0, 0) {$\mathbbm{1}$};
\node (B) at (-1.5,-1) {$a$};
\node(C) at (1.5,-1) {$b$};
\node(D) at (-1.2,-2.5) {$ab$};
\node(E) at (1.2,-2.5) {$ba$};
\node(F) at (0,-4) {$\square$};
\draw[edge,blue,thick] (A) -- (B) node[midway, above] {$a$};
\draw[edge,blue,thick] (A) -- (C) node[midway, above] {$b$};
\path (B) edge[->,thick, bend right = 40] node[midway,left] {$b$} (D);
\path (D) edge[->,thick, bend right = 40] node[midway,right] {$a$} (B);
\path (C) edge[->,thick, bend left = 40] node[midway,right] {$a$} (E);
\path (E) edge[->,thick, bend left = 40] node[midway,left] {$b$} (C);
\draw[edge,blue,thick] (D) -- (F) node[midway, above] {$b$};
\draw[edge,blue,thick] (E) -- (F) node[midway, above] {$a$};
\path (B) edge[->,thick, bend right = 120, blue] node[midway,left] {$a$} (F);
\path (C) edge[->,thick, bend left = 120, blue] node[midway,right] {$b$} (F);
\end{tikzpicture}
\raisebox{1cm}{
\begin{tikzpicture}
\node (A) at (0, 0) {$\mathbbm{1}$};
\node (B) at (-1.5,-1) {$a$};
\node(C) at (1.5,-1) {$b$};
\node(D) at (-1.2,-2.5) {$ab$};
\node(E) at (1.2,-2.5) {$ba$};
\node(F) at (-3,-4) {$aa$};
\node(G) at (-1.2,-4) {$abb$};
\node(H) at (1.2,-4) {$baa$};
\node(I) at (3,-4) {$bb$};
\draw[edge,blue,thick] (A) -- (B) node[midway, above] {$a$};
\draw[edge,blue,thick] (A) -- (C) node[midway, above] {$b$};
\path (B) edge[->,thick, bend right = 40] node[midway,left] {$b$} (D);
\path (D) edge[->,thick, bend right = 40, red, dashed] node[midway,right] {$a$} (B);
\path (C) edge[->,thick, bend left = 40] node[midway,right] {$a$} (E);
\path (E) edge[->,thick, bend left = 40, red, dashed] node[midway,left] {$b$} (C);
\draw[edge,blue,thick] (D) -- (G) node[midway, left] {$b$};
\draw[edge,blue,thick] (E) -- (H) node[midway, right] {$a$};
\path (B) edge[->,thick, bend right = 40, blue] node[midway,left] {$a$} (F);
\path (C) edge[->,thick, bend left = 40, blue] node[midway,right] {$b$} (I);
\end{tikzpicture}}
\end{center}
\caption{\label{figure.B2} \textbf{Left:} The right Cayley graph $\mathsf{RCay}(B(2),\{a,b\})$ with generators
$a=12$ and $b=21$. \textbf{Right:} $\mathsf{Mc}\circ \mathsf{KR}(B(2),\{a,b\}) = \mathsf{KR}(B(2),\{a,b\})$.
Transition edges are indicated in blue.}
\end{figure}
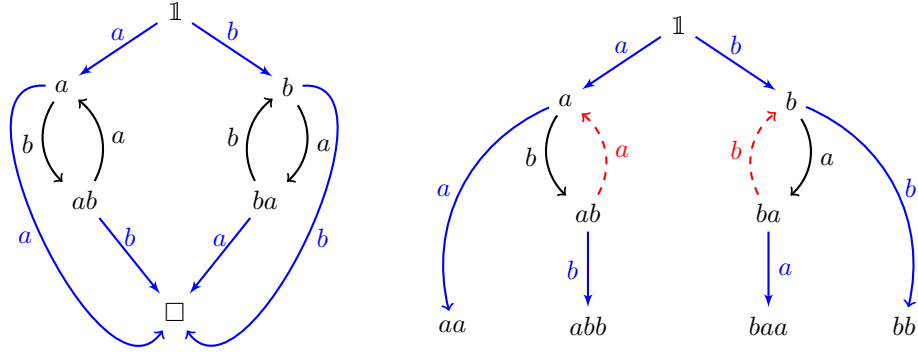

\begin{figure}
\begin{tikzpicture}[->,>=stealth',shorten >=1pt,auto,node distance=2.8cm,
                    semithick]
  \tikzstyle{every state}=[fill=red,draw=none,text=white]

  \node[state]         (A) {$aa$};
  \node[state]         (B) [right of=A] {$baa$};
  \node[state]         (C) [below of =B] {$bb$};
  \node[state]         (D) [below of =A] {$abb$};

  \path (A) edge [loop left] node {$a$} (A)
                 edge [bend left] node {$b$} (B)
           (B) edge node {$b$} (C)
                 edge [bend left] node {$a$} (A)
           (C) edge [loop right] node {$b$} (C)
                 edge [bend left] node {$a$} (D)
           (D) edge node {$a$} (A)
                 edge [bend left] node {$b$} (C);                  
\end{tikzpicture}
\caption{\label{figure.markov chain B2}
Markov chain $\mathcal{M}(\mathsf{KR}(B(2),A))$ of Example~\ref{example.B2}.}
\end{figure}

In general, $\mathsf{RCay}(B(n),A)$ contains a cycle of the form
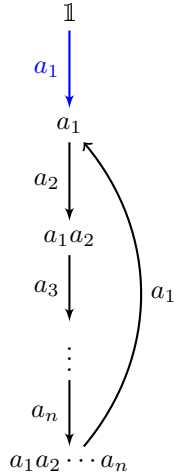
\[
\begin{tikzpicture}[auto]
\node (A) at (0, 0) {$\mathbbm{1}$};
\node (B) at (0,-1.5) {$a_1$};
\node (C) at (0,-3) {$a_1 a_2$};
\node (D) at (0,-4.5) {$\vdots$};
\node (E) at (0,-6) {$a_1 a_2 \cdots a_n$};
\draw[edge,blue,thick] (A) -- (B) node[midway, left] {$a_1$};
\draw[edge,thick] (B) -- (C) node[midway, left] {$a_2$};
\draw[edge,thick] (C) -- (D) node[midway, left] {$a_3$};
\draw[edge,thick] (D) -- (E) node[midway, left] {$a_n$};
\path (E) edge[->,thick, bend right = 40] node[midway,right] {$a_1$} (B);
\end{tikzpicture}
\]
For the vertex $a_1 a_2 \cdots a_{i-1}$, right multiplication by $b$ with $b\neq a_i$ yields $\square$.
There are similar cycles, where each subindex $j$ is replaced by $k+j$ modulo $n$ for a given $0\leqslant k<n$.
Hence the elements in $\mathcal{N}(B(n), A)$ are of the form
\[
	y_{j,i}^k := a_{k+1} a_{k+2} \cdots a_{k+j} a_i, \qquad \qquad
	(0\leqslant k <n, 1\leqslant j \leqslant n)
\]
where all indices are considered modulo $n$ and $i \neq k+j+1$. Note that the McCammond expansion of the
Karnofsky--Rhodes expansion is again stable. 

We have
\[
	\mathsf{NF}^{-1}(y_{j,i}^k) : = a_{k+1} (a_{k+2} a_{k+3} \cdots a_{k+n} a_{k+1})^\star a_{k+2} a_{k+3} \cdots a_{k+j} a_i.
\]
This allows us to compute the stationary distribution of the lumped Markov chain 
$\mathcal{M}(\mathsf{KR}(B(n),A))$ by~\eqref{equation.Psi NF stable}
\[
	\Psi^{\mathsf{KR}(B(n),A)}_{y_{j,i}^k} = \frac{x_{k+1} x_{k+2} \cdots x_{k+j}x_i}{1-x_1 x_2 \cdots x_n},
	\qquad \qquad (0 \leqslant k<n, 1\leqslant j \leqslant n, i\neq k+j+1)
\]
where for simplicity we have set $x_m := x_{a_m}$ for $1\leqslant m \leqslant n$ and again all indices are considered 
modulo $n$.

\subsection{Further cyclic walks -- Rees matrix semigroups}
\label{section.BnZp}

Let $S$ be a semigroup, $I$ and $I'$ be non-empty sets, and $P$ a matrix indexed by $I'$ and $I$ with entries 
$p_{i',i}$ taken from $S$. Then the \defn{Rees matrix semigroup} $(S;I,I';P)$ is the set $I\times S \times I'$ together 
with the multiplication
\[
    (i,s,i')(j,t,j') = (i, s \;p_{i',j}\; t, j').
\]
Similarly, define $(S;I,I';P)^\square$ to be the \defn{Rees matrix semigroup with zero} as the set
$I\times S \times I' \cup \{\square\}$. Here $P$ is an $I'\times I$ matrix with entries in $S\cup \{\square\}$
with multiplication 
\[
    (i,s,i')(j,t,j') = \begin{cases} (i, s \;p_{i',j}\; t, j') & \text{if $p_{i',j} \neq \square$,}\\
    \square & \text{else,} \end{cases}
\]
and $\square$ acts as zero.

Consider the special case of the Rees matrix semigroup $S=(Z_p;[n],[n];\operatorname{id})^\square
$, which consists of the elements $\{(i,g,j) \mid 1\leqslant i,j \leqslant n, g\in Z_p\} \cup \{\square\}$, where $Z_p$ is the 
cyclic group with $p$ elements. The multiplication in this case is given by
\[
	(i,g,j) \cdot (k,g',\ell) = \begin{cases}
	(i,gg',\ell) & \text{if $j=k$,}\\
	\square & \text{otherwise,}
	\end{cases}
\]
and $\square$ acts as zero. Let us choose as generators $A = \{a_i \mid 1\leqslant i \leqslant n\}$, where 
$a_i = (i,\operatorname{id},i+1)$ for  $1\leqslant i < n$ and $a_n = (n,(23\ldots p1),1)$, where 
$(23\ldots p1)$ is the generator of $Z_p$ shifting $i\mapsto i+1$ modulo $p$.
The right Cayley graph and the Karnofsky--Rhodes/McCammond expansion are very similar to the case
of $B(n)$ discussed in Section~\ref{section.Bn}, except that now the cycles have length $np$.
Again $\mathsf{KR}(S,A)$ is stable under the McCammond expansion and the stationary contribution
can easily be computed. We will demonstrate this in the next example.

\begin{example}
\label{example.B2Z2}
Consider the special case of $S=(Z_2;[2],[2]; \begin{pmatrix}1&0\\0&1 \end{pmatrix})^\square$ with generators $A=\{a,b\}$, 
where $a=(1,\operatorname{id},2)$ and $b=(2,(21),1)$. The right Cayley graph and $\mathsf{Mc}\circ\mathsf{KR}(S,A)$
are similar to those in Figure~\ref{figure.B2}, except that the cycles now contain the elements $a,ab,aba,abab$ 
and $b,ba,bab,baba$, respectively. The minimal ideal is $K(S)=\{\square\}$ and the normal forms are given by
\[
	\mathcal{N}(S,A) = \{aa,abb,abaa,ababb,bb,baa,babb,babaa\}.
\]
The Markov chain $\mathcal{M}(\mathsf{KR}(S,A))$ is depicted in Figure~\ref{figure.markov chain B2Z2}.
To compute the stationary distribution, we first obtain
\begin{equation*}
\begin{split}
	\mathsf{NF}^{-1}(aa) = a (baba)^\star a, \qquad \qquad &\mathsf{NF}^{-1}(baa) = b (abab)^\star aa,\\
	\mathsf{NF}^{-1}(abaa) = a (baba)^\star baa, \qquad \qquad & \mathsf{NF}^{-1}(babaa) = b(abab)^\star abaa,
\end{split}
\end{equation*}
and similarly with the letters $a$ and $b$ interchanged everywhere.
Hence by~\eqref{equation.Psi NF stable} we obtain the stationary distribution
\begin{equation*}
	\Psi^{\mathsf{KR}(S,A)}_{aa} = \frac{x_a^2}{1-x^2_a x^2_b}, \quad 
	\Psi^{\mathsf{KR}(S,A)}_{baa} = \frac{x^2_a x_b}{1-x^2_a x^2_b}, \quad
	\Psi^{\mathsf{KR}(S,A)}_{abaa} = \frac{x_a^3x_b}{1-x^2_a x^2_b}, \quad  
	\Psi^{\mathsf{KR}(S,A)}_{babaa} = \frac{x_a^3 x^2_b}{1-x^2_a x^2_b},
\end{equation*}
and similarly with the letters $a$ and $b$ interchanged everywhere. It can be verified that the stationary distributions
add up to one, confirming~\eqref{equation.prob condition}.
\end{example}

\begin{figure}
\begin{tikzpicture}[->,>=stealth',shorten >=1pt,auto,node distance=2.8cm,
                    semithick]
  \tikzstyle{every state}=[fill=red,draw=none,text=white]

  \node[state]         (A) {$aa$};
  \node[state]         (B) [right of=A] {$baa$};
  \node[state]          (C) [right of=B] {$abaa$};
  \node[state]         (D) [right of=C] {$babaa$};
  \node[state]         (Ap) [below of =A] {$bb$};
  \node[state]         (Bp) [right of =Ap] {$abb$};
  \node[state]          (Cp) [right of=Bp] {$babb$};
  \node[state]         (Dp) [right of=Cp] {$ababb$};

  \path (A) edge [loop left] node {$a$} (A)
                 edge node {$b$} (B)
           (B) edge node {$a$} (C)
                 edge node [above=0.15cm]{$\quad b$} (Ap)
           (C) edge node {$b$} (D)
                 edge [bend right] node [above] {$a$} (A)
           (D) edge node [above=0.3cm] {$\qquad \qquad b$} (Ap)
                 edge [bend right=38] node [above] {$a$} (A)
            (Ap) edge [loop left] node {$b$} (Ap)
                 edge node [below] {$a$} (Bp)
           (Bp) edge node [below] {$b$} (Cp)
                 edge node [below=0.15cm] {$\quad a$} (A)
           (Cp) edge node [below] {$a$} (Dp)
                 edge [bend left] node {$b$} (Ap)
           (Dp) edge node [below=0.3cm]{$\qquad \qquad a$} (A)
                 edge [bend left=38] node {$b$} (Ap);                              
\end{tikzpicture}
\caption{\label{figure.markov chain B2Z2}
Markov chain $\mathcal{M}(\mathsf{KR}(S,A))$ of Example~\ref{example.B2Z2}.}
\end{figure}

\begin{example}
\label{example.Rees}
Consider the Rees matrix semigroup $S=(Z_2;[2],[2];\begin{pmatrix} 1&1\\1&-1\end{pmatrix})$
with generators $A=\{a,b\}$ with $a=(1,1,2)$ and $b=(2,1,1)$. The Markov chain
$\mathcal{M}(S,A)$ is not irreducible, see Figure~\ref{figure.markov chain Rees}.
Since the minimal ideal $K(S)$ is not left zero, we need to apply Corollary~\ref{corollary.stationary general}.
To this end, we consider $S'=S\cup \{\square\}$ and $A'=A\cup \{\square\}$. The normal forms are
\[
	\mathcal{N}(S',A') = \{\square, a\square, ab\square, aba\square, abab\square, b\square, ba \square, bab\square,
	baba\square\}.
\]
We have
\begin{equation*}
\begin{split}
	\mathsf{NF}^{-1}(\square) &=\square,\\
	\mathsf{NF}^{-1}(a\square) &=aa^\star (bb^\star a a^\star b b^\star a a^\star)^\star \square,\\
	\mathsf{NF}^{-1}(ab\square) &=aa^\star (bb^\star a a^\star b b^\star a a^\star)^\star bb^\star\square,\\
	\mathsf{NF}^{-1}(aba\square) &=aa^\star (bb^\star a a^\star b b^\star a a^\star)^\star bb^\star a a^\star \square,\\
	\mathsf{NF}^{-1}(abab\square) &=aa^\star (bb^\star a a^\star b b^\star a a^\star)^\star bb^\star a a^\star b b^\star \square,
\end{split}
\end{equation*}
and similarly with $a$ and $b$ interchanged. We obtain
\begin{equation*}
\begin{aligned}
	\Psi_\square^{(S',A')} &= x_\square && \stackrel{x_\square \to 0}{\longrightarrow} \quad 0,\\
	\Psi_{a\square}^{(S',A')} &= \frac{x_a x_\square}{(1-x_a)(1-\frac{x_a^2 x_b^2}{(1-x_a)^2 (1-x_b)^2})}
	= \frac{x_a (1-x_a)(1-x_b)^2}{x_\square+2x_ax_b} 
	&& \stackrel{x_\square \to 0}{\longrightarrow} \quad \frac{1}{2}x_a^2,\\
	\Psi_{ab\square}^{(S',A')} &= \frac{x_a x_b (1-x_a)(1-x_b)}{x_\square+2x_ax_b} 
	&& \stackrel{x_\square \to 0}{\longrightarrow} \quad \frac{1}{2}x_a x_b,\\
	\Psi_{aba\square}^{(S',A')} &= \frac{x^2_a x_b (1-x_b)}{x_\square+2x_ax_b} 
	&& \stackrel{x_\square \to 0}{\longrightarrow} \quad \frac{1}{2}x^2_a,\\
	\Psi_{abab\square}^{(S',A')} &= \frac{x^2_a x^2_b}{x_\square+2x_ax_b} 
	&& \stackrel{x_\square \to 0}{\longrightarrow} \quad \frac{1}{2}x_a x_b,\\
\end{aligned}
\end{equation*}
and similarly with $a$ and $b$ interchanged.
\end{example}

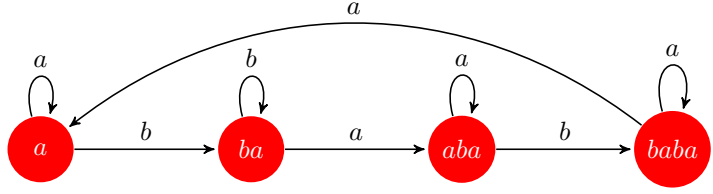
\begin{figure}
\begin{tikzpicture}[->,>=stealth',shorten >=1pt,auto,node distance=2.8cm,
                    semithick]
  \tikzstyle{every state}=[fill=red,draw=none,text=white]

  \node[state]         (A) {$a$};
  \node[state]         (B) [right of=A] {$ba$};
  \node[state]          (C) [right of=B] {$aba$};
  \node[state]         (D) [right of=C] {$baba$};

  \path (A) edge [loop above] node {$a$} (A)
                 edge node {$b$} (B)
           (B) edge [loop above] node {$b$} (B)
                 edge node {$a$} (C)
           (C) edge [loop above] node {$a$} (C)
                 edge node {$b$} (D)
           (D) edge [loop above] node {$a$} (D)
                 edge [bend right=38] node [above] {$a$} (A);
\end{tikzpicture}
\caption{\label{figure.markov chain Rees}
Part of the Markov chain $\mathcal{M}(\mathsf{KR}(S,A))$ of Example~\ref{example.Rees}.
The other part is exactly the same with $a$ and $b$ interchanged everywhere. This chain is not irreducible.
}
\end{figure}

\subsection{Markov chains on $\mathscr{R}$-trivial monoids}
\label{section.R trivial}

Markov chains for $\mathscr{R}$-trivial monoids were studied in detail in~\cite{ASST.2015}.
In particular, the eigenvalues of the transition matrix and their multiplicities, the stationary distributions, and bounds on 
the mixing times were derived for general $\mathscr{R}$-trivial monoids. This class of Markov chains contains a vast 
number of previously studied Markov chains, such as the Tsetlin library~\cite{Tsetlin.1963,Hendricks.1972, Hendricks.1973},
walks on hyperplane arrangements~\cite{Bidigare.1997,BD.1998,BHR.1999}, Brown's generalization to left regular 
bands~\cite{Brown.2000}, edge flipping in graphs~\cite{ChungGraham.2012}, random walks on linear extensions
of a poset~\cite{AKS.2014}, and others~\cite{AS.2010, Ayyer.2011,AS.2013,ASST.2015a,ASST.2015}.

We will now show how the stationary distribution of a Markov chain for an $\mathscr{R}$-trivial monoid
as given in~\cite[Theorem 4.12]{ASST.2015} can be derived using the methods of this paper.

Recall that for a semigroup $S$, two elements $s,s'\in S$ are in the same $\mathscr{R}$-class if the corresponding right 
ideals are equal, that is, $s S = s'S$. We say that $S$ is \defn{$\mathscr{R}$-trivial} if all $\mathscr{R}$-classes are trivial, 
meaning that $s S = s' S$ implies that $s=s'$.

The Karnofsky--Rhodes expansion of the right Cayley graph of an $\mathscr{R}$-trivial semigroup 
$(S,A)$ is a directed tree after the removal of loop edges and $\mathsf{KR}(S,A)$ is stable under the McCammond
expansion. If $w=w_1 w_2 \cdots w_\ell \in \mathcal{N}(S,A)$, then
\[
	\mathsf{NF}^{-1}(w) = w_1 N_1^\star w_2 N_2^\star \cdots w_{\ell-1} N_{\ell-1}^\star w_\ell,
\]
where $N_i = N_i^w = \{ a\in A \mid [w_1 \cdots w_i a]_S = [w_1 \cdots w_i]_S\}$ is the set of generators that stabilize
the element $[w_1\cdots w_i]_S$. We set $N_\ell = \emptyset$. Hence the stationary distribution of 
$\mathcal{M}(\mathsf{KR}(S,A))$ is
\[
	\Psi^{\mathsf{KR}(S,A)}_w = \prod_{i=1}^\ell \frac{x_{w_i}}{1-\sum_{a\in N_i} x_a} 
	\qquad \text{for $w = w_1 \cdots w_\ell\in \mathcal{N}(S,A)$}
\]
in agreement with~\cite[Corollary 4.13]{ASST.2015}.

If the Karnofsky--Rhodes expansion is not the same as the original right Cayley graph, one can in fact lump
the Markov chain $\mathcal{M}(\mathsf{KR}(S,A))$ further to $\mathcal{M}(S,A)$ by applying $\varphi \colon (A^+,A) \to 
(S,A)$ of~\eqref{equation.phi}. Recall from Definition~\ref{definition.NF inv} that $\mathsf{Red}(s)$ is the set of all elements 
$w \in \mathcal{N}(S,A)$ such that $[w]_S = s$. For this lumped Markov chain, we obtain the stationary distribution
by~\eqref{equation.Psi NF inv}
\[
	\Psi_s = \sum_{w \in \mathsf{Red}(s)} \prod_{i=1}^{\ell(w)} \frac{x_{w_i}}{1-\sum_{a\in N_i^w} x_a}
	\qquad \text{for $s\in S$},
\]
where $\ell(w)$ is the length of the word $w$, in agreement with~\cite[Theorem 4.12]{ASST.2015}.

\subsection{Adding constants: The $\mathsf{bar}$ operation}
\label{section.bar}

In this and the next section, we use the two operations $\mathsf{bar}$ and $\flat$ introduced in 
Section~\ref{section.bar and flat} (see also \cite{LRS.2017}) to produce new Markov chains from known examples. 

We use two \defn{stability conditions}. The first one is stability under the McCammond expansion:
\begin{equation}
\label{equation.stable}
	\mathsf{Mc} \circ \mathsf{KR}(S,A) = \mathsf{KR}(S,A).
\end{equation}
The second stability condition is stability under both the Karnofsky--Rhodes and McCammond expansion:
\begin{equation}
\label{equation.stable1}
	\mathsf{Mc} \circ \mathsf{KR}(S,A) = (S,A).
\end{equation}
Since $\mathsf{KR}^2=\mathsf{KR}$, stability condition~\eqref{equation.stable1}
implies~\eqref{equation.stable}, but not vice versa. For example, the semigroup $S=\{0,1\}$ with generators $A=\{0,1\}$ 
of Example~\ref{example.01} satisfies~\eqref{equation.stable}, but not~\eqref{equation.stable1}.

\begin{conjecture}
$\mathsf{KR}(S,A)$ is stable under $\mathsf{Mc}$ if and only if $\mathsf{Mc} \circ \mathsf{KR} (S,A)$ is a right 
Cayley graph. In other words, if $\mathsf{Mc}$ changes any of the $\mathscr{R}$-classes of $\mathsf{KR}(S,A)$, 
then $\mathsf{Mc} \circ \mathsf{KR} (S,A)$ cannot be a right Cayley graph.
\end{conjecture}

\begin{proposition}
\label{proposition.bar stable}
Suppose that $(S,A)$ satisfies the stability condition~\eqref{equation.stable1}. Then $(S,A)^{\mathsf{bar}}$ 
satisfies the stability condition~\eqref{equation.stable}.
\end{proposition}

\begin{proof}
By assumption~\eqref{equation.stable1}, $(S,A)$ is stable under both $\mathsf{KR}$ and $\mathsf{Mc}$. Under the 
$\mathsf{bar}$ construction, $\mathsf{RCay}((S,A)^\mathsf{bar})$ is obtained from $\mathsf{RCay}(S,A)$ by adding a new 
edge labeled $\overline{\mathbbm{1}}$ from each vertex to a new vertex labeled $\overline{\mathbbm{1}}$. Underneath 
the vertex $\overline{\mathbbm{1}}$, there is a copy of $\mathsf{RCay}(S,A)$, where each vertex 
$x \in V(\mathsf{RCay}(S,A))$ is replaced by $\overline{x}$ thanks to the relation $\overline{x} \cdot a = \overline{x\cdot a}$
for any $a \in A$. In addition, each vertex below $\overline{\mathbbm{1}}$ has an edge labeled $\overline{\mathbbm{1}}$ 
looping back to vertex $\overline{\mathbbm{1}}$. Since the original $\mathsf{RCay}(S,A)$ was stable under the 
Karnofsky--Rhodes and McCammond expansion, the effect of the Karnofksy--Rhodes expansion of $(S,A)^\mathsf{bar}$
is to have a separate vertex $\overline{\mathbbm{1}}$ below each vertex, which is stable under the McCammond expansion,
proving the claim.
\end{proof}

\begin{figure}
\begin{center}
\begin{tikzpicture}[auto]
\node (A) at (0, 0) {$\mathbbm{1}$};
\node (B) at (-3, -1.5) {$1$};
\node (C) at (3, -1.5) {$2$};
\node (D) at (-3, -3) {$12$};
\node (E) at (3,-3) {$21$};
\node[circle,fill=black,text=white] (G1) at (0,-3) {$G$};
\node[circle,fill=black,text=white] (G2) at (-6,-3) {$G$};
\node[circle,fill=black,text=white] (G3) at (6,-3) {$G$};
\node[circle,fill=black,text=white] (G4) at (-3,-4.5) {$G$};
\node[circle,fill=black,text=white] (G5) at (3,-4.5) {$G$};

\draw[edge,blue,thick] (A) -- (B) node[midway, left] {$1$};
\draw[edge,blue,thick] (A) -- (C) node[midway, right] {$2$};
\draw[edge,blue,thick] (B) -- (D) node[midway, left] {$2$};
\draw[edge,blue,thick] (C) -- (E) node[midway, right] {$1$};
\draw[edge,blue,thick] (A) -- (G1) node[midway, left] {$\overline{\mathbbm{1}}$};
\draw[edge,blue,thick] (B) -- (G2) node[midway, above] {$\overline{\mathbbm{1}}$};
\draw[edge,blue,thick] (C) -- (G3) node[midway, above] {$\overline{\mathbbm{1}}$};
\draw[edge,blue,thick] (D) -- (G4) node[midway, left] {$\overline{\mathbbm{1}}$};
\draw[edge,blue,thick] (E) -- (G5) node[midway, right] {$\overline{\mathbbm{1}}$};
\path
	(B) edge [loop right, red, dashed,thick] node {$1$} (B)
	(C) edge [loop left, red, dashed,thick] node {$2$} (C)
	(D) edge [loop right, red, dashed, thick] node {$1,2$} (D)
	(E) edge [loop left, red, dashed, thick] node {$1,2$} (E);
	
\node (a) at (0,-6) {$\overline{\mathbbm{1}}$};
\node (b) at (-1.5,-7.5) {$\overline{1}$};
\node (c) at (1.5,-7.5) {$\overline{2}$};
\node (d) at (-1.5,-9) {$\overline{12}$};
\node (e) at (1.5,-9) {$\overline{21}$};
\node[circle,fill=black,text=white] (g) at (-3.5,-7.5) {$G$};
\node (gg) at (-2.9,-7.5) {$=$};
\draw[edge,thick] (a) -- (b) node[midway, left] {$1$};
\draw[edge,thick] (a) -- (c) node[midway, right] {$2$};
\draw[edge,thick] (b) -- (d) node[midway, left] {$2$};
\draw[edge,thick] (c) -- (e) node[midway, right] {$1$};
\path
	(b) edge [loop right, red, dashed,thick] node {$1$} (b)
	(c) edge [loop left, red, dashed,thick] node {$2$} (c)
	(d) edge [loop right, red, dashed, thick] node {$1,2$} (d)
	(e) edge [loop left, red, dashed, thick] node {$1,2$} (e);
\path (b) edge[->,thick, bend left=30, dashed, red] node[midway,left] {$\overline{\mathbbm{1}}$} (a);
\path (c) edge[->,thick, bend right=30, dashed, red] node[midway,right] {$\overline{\mathbbm{1}}$} (a);
\path (d) edge[->,thick, bend left=90, dashed, red] node[midway,left] {$\overline{\mathbbm{1}}$} (a);
\path (e) edge[->,thick, bend right=90, dashed, red] node[midway,right] {$\overline{\mathbbm{1}}$} (a);
\end{tikzpicture}
\end{center}
\caption{The Karnofsky--Rhodes expansion of $\left(\mathsf{KR}(P(2),[2])\right)^\mathsf{bar}$.
\label{figure.bar}
}
\end{figure}

\begin{example}
Let us consider $(S,A)=\mathsf{KR}(P(2),[2])$, which satisfies the stability condition~\eqref{equation.stable1}.
Then the Karnofsky--Rhodes expansion of $(S,A)^\mathsf{bar} = \left(\mathsf{KR}(P(2),[2])\right)^\mathsf{bar}$ is given 
in Figure~\ref{figure.bar}. It can easily be checked that it is stable under the McCammond expansion, so that
$(S,A)^\mathsf{bar}$ satisfies~\eqref{equation.stable} verifying Proposition~\ref{proposition.bar stable}.
\end{example}

\begin{remark}
\label{remark.add zero}
Note that if $(S,A)$ satisfies~\eqref{equation.stable}, then $(S\cup \{\square\}, A\cup \{\square\})$, where $\square$ is
a new zero element (that is $\square \cdot x = x \cdot \square = \square$ for all $x \in S$), also 
satisfies~\eqref{equation.stable}.
\end{remark}

\begin{example}
Consider the semigroup $N_n = \langle a \mid a^n = 0 \rangle = \{0,a,a^2,\ldots,a^{n-1}\}$ with generator $\{a\}$.
Its right Cayley graph is a line with $n+1$ vertices starting at $\mathbbm{1}$ with intermediate vertices $a^i$ 
($1\leqslant i < n$) and ending in $0=a^n$, with a loop at $a^n$. Hence it satisfies the stability 
condition~\eqref{equation.stable1}. The Karnofsky--Rhodes expansion of $(N_n,\{a\})^\mathsf{bar}$ is the previous 
right Cayley graph with a string of length $n+1$ attached to each previous vertex starting at $\overline{\mathbbm{1}}$ 
with intermediate vertices $\overline{a}^i$ ($1\leqslant i < n$) and ending in $\overline{a}^n=\overline{0}=0$. In addition, 
there is an edge labeled $\overline{\mathbbm{1}}$ going from each vertex $\overline{a}^i$ ($0\leqslant i \leqslant n$) to 
vertex $\overline{\mathbbm{1}}$. This graph has the unique path property and hence $(N_n,\{a\})^\mathsf{bar}$ 
satisfies~\eqref{equation.stable} confirming Proposition~\ref{proposition.bar stable}.

Let us now add a new zero $\square$ as a generator, so that $A=\{a,\overline{\mathbbm{1}},\square\}$ and
$S = (N_n,\{a\})^\mathsf{bar} \cup \{\square\}$. By Remark~\ref{remark.add zero}, $(S,A)$ also 
satisfies~\eqref{equation.stable}. The minimal ideal is $K(S)=\{\square\}$. Then the normal forms are
\[
	\mathcal{N}(S,A) = \{ a^i \square, a^i  \overline{\mathbbm{1}} a^j \square \mid 0\leqslant i,j \leqslant n\}.
\]
Furthermore
\begin{equation*}
\begin{split}
	\mathsf{NF}^{-1}(a^i \square) &= a^i \square, \qquad \qquad \qquad \;(0\leqslant i < n)\\
	\mathsf{NF}^{-1}(a^n \square) &= a^n a^\star \square,\\ 
	\mathsf{NF}^{-1}(a^i  \overline{\mathbbm{1}} a^j \square) &= a^i \overline{\mathbbm{1}}
	(a^\star \overline{\mathbbm{1}})^\star a^j \square, \qquad (0\leqslant i,j<n)\\
	\mathsf{NF}^{-1}(a^n  \overline{\mathbbm{1}} a^j \square) &= a^n a^\star \overline{\mathbbm{1}}
	(a^\star \overline{\mathbbm{1}})^\star a^j \square, \quad (0\leqslant j<n)\\
	\mathsf{NF}^{-1}(a^i  \overline{\mathbbm{1}} a^n \square) &= a^i \overline{\mathbbm{1}}
	(a^\star \overline{\mathbbm{1}})^\star a^n a^\star \square, \quad (0\leqslant i<n)\\
	\mathsf{NF}^{-1}(a^n  \overline{\mathbbm{1}} a^n \square) &= a^n a^\star \overline{\mathbbm{1}}
	(a^\star \overline{\mathbbm{1}})^\star a^n a^\star \square.
\end{split}
\end{equation*}
It is clear in this case, that each word in the Kleene expressions occurs at most once.
Assigning probabilities $x_a,x_{\overline{\mathbbm{1}}},x_\square$ to the three generators in $A=\{a,\overline{\mathbbm{1}},
\square\}$, we hence obtain the stationary distribution of the lumped Markov chain $\mathcal{M}(S,A)$
\begin{equation*}
\begin{split}
	\Psi^{(S,A)}_{a^i \square} &= x_a^i x_{\square}, \qquad \qquad \qquad (0\leqslant i < n)\\
	\Psi^{(S,A)}_{a^n \square} &= \frac{x_a^n x_{\square}}{1-x_a},\\
	\Psi^{(S,A)}_{a^i \overline{\mathbbm{1}} \overline{a}^j \square} &= \frac{x_a^{i+j} x_{\overline{\mathbbm{1}}} x_\square}
	{1-\frac{x_{\overline{\mathbbm{1}}}}{1-x_a}}, \qquad \qquad (0\leqslant i,j<n)\\
	\Psi^{(S,A)}_{a^n \overline{\mathbbm{1}} \overline{a}^j \square} &= \frac{x_a^{n+j} x_{\overline{\mathbbm{1}}} x_\square}
	{1-x_a-x_{\overline{\mathbbm{1}}}}, \qquad \quad (0\leqslant j<n)\\
	\Psi^{(S,A)}_{a^i \overline{\mathbbm{1}} \overline{a}^n \square} &= \frac{x_a^{i+n} x_{\overline{\mathbbm{1}}} x_\square}
	{1-x_a-x_{\overline{\mathbbm{1}}}}, \qquad \quad (0\leqslant i<n)\\
	\Psi^{(S,A)}_{a^n \overline{\mathbbm{1}} \overline{a}^n \square} &= \frac{x_a^{2n} x_{\overline{\mathbbm{1}}} x_\square}
	{(1-x_a)(1-x_a-x_{\overline{\mathbbm{1}}})}.
\end{split}
\end{equation*}
Using $\sum_{i=0}^{n-1} x_a^i = \frac{1-x_a^n}{1-x_a}$ it can easily be verified that the stationary probabilities
add up to one.
\end{example}

Now let us consider a general finite $A$-semigroup $(S,A)$ that satisfies~\eqref{equation.stable1}. 
By Proposition~\ref{proposition.bar stable} and Remark~\ref{remark.add zero}, $(S',A')=((S,A)^\mathsf{bar}\cup \{\square\}, 
A\cup \{\overline{\mathbbm{1}},\square\})$ satisfies~\eqref{equation.stable} and hence yields a Markov chain
$\mathcal{M}(S',A')$ with minimal ideal $K(S')=\{\square\}$. The normal forms $\mathcal{N}(S',A')$ are given by
\[
	w \square, \; w \overline{\mathbbm{1}} w' \quad \text{for $w,w'\in \mathcal{N}$,}
\]
where $\mathcal{N} := \mathcal{N}(S\cup\{\square\},A\cup\{\square\},\{\square\}) \setminus \{\square\}$, where the removal
of $\square$ means the normal forms in $\mathcal{N}(S\cup\{\square\},A\cup\{\square\},\{\square\})$ without the last $\square$.
If the Kleene expressions for $\mathsf{NF}^{-1}(w)$ for $w\in \mathcal{N}$ are known, then the Kleene expressions for 
the normal forms in $\mathcal{N}(S',A')$ can also be derived:
\begin{equation}
\label{equation.bar normal}
\begin{split}
	\mathsf{NF}^{-1}(w\square) &= \mathsf{NF}^{-1}(w) \square,\\
	\mathsf{NF}^{-1}(w \overline{\mathbbm{1}} w' \square) &= \mathsf{NF}^{-1}(w) \overline{\mathbbm{1}}
	\left( \cup_{v\in \mathcal{N}} \mathsf{NF}^{-1}(v) 
	\overline{\mathbbm{1}} \right)^\star \mathsf{NF}^{-1}(w') \square.
\end{split}
\end{equation}

\begin{remark}
Intuitively, one can interpret the $\mathsf{bar}$ operation as reproduction: each cell (or vertex in the Cayley graph)
produces a copy of itself (with edges back to its origin $\overline{\mathbbm{1}}$). Equivalently, applying $\mathbbm{1}$
means to reset the library to the beginning.
By Proposition~\ref{proposition.bar stable}, one can repeat the operation $\mathsf{KR} \circ \mathsf{bar}$ an arbitrary
number of times. Repeating it $n$ times and letting $n$ tend to infinity, has the flavor of a fractal:
in any portion of the graph, one can zoom in and find the original right Cayley graph, or in fact $(\mathsf{KR} \circ 
\mathsf{bar})^k(S,A)$ for any $k>0$.
\end{remark}

\subsection{The $\flat$ operation}
\label{section.flat}

Recall the flat operation $\flat$ from Section~\ref{section.bar and flat}.

\begin{proposition}
\label{proposition.flat stable}
Suppose that $(S,A)$ satisfies the stability condition~\eqref{equation.stable}. Then $(S,A)^\flat$ also
satisfies the stability condition~\eqref{equation.stable}.
\end{proposition}

\begin{proof}
The right Cayley graph $\mathsf{RCay}((S,A)^\flat)$ can be obtained from $\mathsf{RCay}(S,A)$ by adding
to each vertex $x$ a new edge labeled $\widetilde{\mathbbm{1}}$ to $\widetilde{x}$. Observe that this
implies that $\mathsf{KR}$ commutes with $\flat$. By assumption, the Karnofsky--Rhodes expansion $\mathsf{KR}(S,A)$ 
has the unique path property (since it is stable under the McCammond expansion by~\eqref{equation.stable}). Since 
$\mathsf{KR}$ and $\flat$ commute, $\mathsf{KR}((S,A)^\flat)$ is obtained from $\mathsf{KR}(S,A)$ by adding a new edge
labeled $\widetilde{\mathbbm{1}}$ to each vertex $x$ which goes to $\widetilde{x}$, which is a trivial one point 
$\mathscr{R}$-class. Since $\mathsf{KR}(S,A)$ has the unique path property so does $\mathsf{KR}((S,A)^\flat)$, proving 
the claim.
\end{proof}

\begin{example}
Consider $(P(2),[2])$. Recall that $(P(n),[n])$ yields the Tsetlin library (see Section~\ref{section.Tsetlin library}) and 
satisfies the stability condition~\eqref{equation.stable}. The Karnofsky--Rhodes expansion of $(P(2),[2])^\flat$ is depicted 
below:
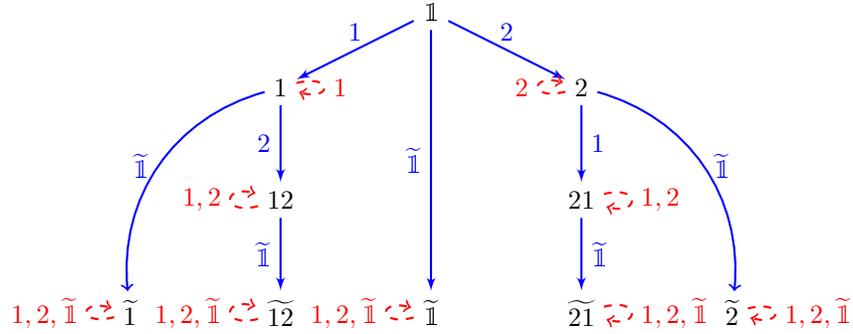
\begin{center}
\begin{tikzpicture}
\node (A) at (0, 0) {$\mathbbm{1}$};
\node (B) at (-2,-1) {$1$};
\node(C) at (2,-1) {$2$};
\node(D) at (-2,-2.5) {$12$};
\node(E) at (2,-2.5) {$21$};
\node(F) at (-4,-4) {$\widetilde{1}$};
\node(G) at (-2,-4) {$\widetilde{12}$};
\node(H) at (2,-4) {$\widetilde{21}$};
\node(I) at (4,-4) {$\widetilde{2}$};
\node(J) at (0,-4) {$\widetilde{\mathbbm{1}}$};
\draw[edge,blue,thick] (A) -- (B) node[midway, above] {$1$};
\draw[edge,blue,thick] (A) -- (C) node[midway, above] {$2$};
\draw[edge,blue,thick] (B) -- (D) node[midway, left] {$2$};
\draw[edge,blue,thick] (C) -- (E) node[midway, right] {$1$};
\draw[edge,blue,thick] (D) -- (G) node[midway, left] {$\widetilde{\mathbbm{1}}$};
\draw[edge,blue,thick] (E) -- (H) node[midway, right] {$\widetilde{\mathbbm{1}}$};
\path (B) edge[->,thick, bend right = 40, blue] node[midway,left] {$\widetilde{\mathbbm{1}}$} (F);
\path (C) edge[->,thick, bend left = 40, blue] node[midway,right] {$\widetilde{\mathbbm{1}}$} (I);
\draw[edge,blue,thick] (A) -- (J) node[midway, left] {$\widetilde{\mathbbm{1}}$};
\path
	(B) edge [loop right, red, dashed,thick] node {$1$} (B)
	(C) edge [loop left, red, dashed, thick] node {$2$} (C)
	(D) edge [loop left, red, dashed,thick] node {$1,2$} (D)
	(E) edge [loop right, red, dashed,thick] node {$1,2$} (E)
	(F) edge [loop left, red, dashed,thick] node {$1,2,\widetilde{\mathbbm{1}}$} (F)
	(G) edge [loop left, red, dashed,thick] node {$1,2,\widetilde{\mathbbm{1}}$} (G)
	(J) edge [loop left, red, dashed,thick] node {$1,2,\widetilde{\mathbbm{1}}$} (J)
	(H) edge [loop right, red, dashed,thick] node {$1,2,\widetilde{\mathbbm{1}}$} (H)
	(I) edge [loop right, red, dashed,thick] node {$1,2,\widetilde{\mathbbm{1}}$} (I);
\end{tikzpicture}
\end{center}
This expansion indeed has the unique path property, so that stability condition~\eqref{equation.stable}
holds. This verifies Proposition~\ref{proposition.flat stable}.
\end{example}

Propositions~\ref{proposition.bar stable} and~\ref{proposition.flat stable} allow us to construct an infinite
tower of semigroups satisfying one of the stability conditions from a given semigroup.

\begin{theorem}
\label{theorem.bar flat}
Suppose $(S,A)$ satisfies stability condition~\eqref{equation.stable}. Then
\begin{equation*}
\begin{split}
	\left(\mathsf{bar} \circ \mathsf{KR} \circ \flat\right)^n (S,A) \qquad 
	&\text{satisfies stability condition~\eqref{equation.stable} for any $n\geqslant 0$,}\\
	\flat \circ \left(\mathsf{bar} \circ \mathsf{KR} \circ \flat\right)^n (S,A) \qquad 
	&\text{satisfies stability condition~\eqref{equation.stable} for any $n\geqslant 0$,}\\
	\mathsf{KR}\circ \flat \circ \left(\mathsf{bar} \circ \mathsf{KR} \circ \flat\right)^n (S,A) \qquad 
	&\text{satisfies stability condition~\eqref{equation.stable1} for any $n\geqslant 0$.}
\end{split}
\end{equation*}
Similarly, if $(S,A)$ satisfies stability condition~\eqref{equation.stable1}, then
\begin{equation*}
\begin{split}
	\left(\mathsf{KR} \circ \flat \circ \mathsf{bar}\right)^n (S,A) \qquad 
	&\text{satisfies stability condition~\eqref{equation.stable1} for any $n\geqslant 0$,}\\
	\mathsf{bar}  \circ \left(\mathsf{KR} \circ \flat \circ \mathsf{bar}\right)^n (S,A) \qquad 
	&\text{satisfies stability condition~\eqref{equation.stable} for any $n\geqslant 0$,}\\
	\flat \circ \mathsf{bar}  \circ \left(\mathsf{KR} \circ \flat \circ \mathsf{bar}\right)^n (S,A) \qquad
	&\text{satisfies stability condition~\eqref{equation.stable} for any $n\geqslant 0$.}
\end{split}
\end{equation*}
\end{theorem}

\begin{proof}
This follows directly from Propositions~\ref{proposition.bar stable} and~\ref{proposition.flat stable} and the fact
that $\mathsf{KR}^2 = \mathsf{KR}$.
\end{proof}

\begin{example}
Let us now compute the normal forms and Kleene expressions for
\[
	(S,A) = \flat \circ \mathsf{KR} \circ \mathsf{bar} \circ \mathsf{KR}(P(n),[n]).
\]
Recall that $\mathsf{KR}$ and $\flat$ commute. The Cayley graph for $\mathsf{KR} \circ \mathsf{bar} \circ 
\mathsf{KR}(P(2),[2])$ is given in Figure~\ref{figure.bar}. Let $\mathcal{N}$ denote the set of normal forms of 
$\flat \circ \mathsf{KR}(P(n),[n])$ with the last $\widetilde{\mathbbm{1}}$ removed. The normal forms of $(S,A)$ are 
given by
\begin{equation*}
	w \widetilde{\mathbbm{1}}, \quad 
	w \overline{\mathbbm{1}} w' \widetilde{\mathbbm{1}} \qquad \text{for $w,w' \in \mathcal{N}$.}
\end{equation*}
Then $\mathsf{NF}^{-1}(w \widetilde{\mathbbm{1}})$ and $\mathsf{NF}^{-1}(w \overline{\mathbbm{1}} w' 
\widetilde{\mathbbm{1}})$ are given by the same formula as in~\eqref{equation.bar normal} with $\square$
replaced by $\widetilde{\mathbbm{1}}$. The normal forms in $\mathcal{N}$ are all words without repeated letters
in the alphabet $[n]$. For a word $w=w_1 w_2 \ldots w_k \in \mathcal{N}$, we have
\[
	\mathsf{NF}^{-1}(w) = w_1 w_1^\star w_2 \{w_1,w_2\}^\star \cdots w_k \{w_1,\ldots,w_k\}^\star.
\]
It is clear in this example that no words are repeated in the Kleene expressions. Hence the stationary distribution
follows directly by applying~\eqref{equation.geometric}.
\end{example}

\subsection{Burnside examples}
\label{section.burnside}

The Burnside semigroups with generators in $A$ are the infinite semigroups of the form
\[
	\mathcal{B}(m,n) = \langle A \mid t^m = t^{m+n} \quad \forall t\in A^+\rangle.
\]
For $m\geqslant 6$ and $n\geqslant 1$, McCammond~\cite{McCammond.1991} showed that the Burnside 
semigroups are finite $\mathscr{J}$-above, have a decidable word problem, and their maximal subgroups are cyclic.
These results were generalized by de Luca and Varricchio~\cite{Luca.Varricchio.1992}, Guba~\cite{Guba.1993,
Guba.1993a}, and do Lago~\cite{Lago.1996} to $m\geqslant 3$ and $n\geqslant 1$.
Recall that $\mathscr{J}$-order in a semigroup $S$ is defined by $s\geqslant_{\mathscr{J}} s'$ if 
$s'=s$, $s'=xs$, $s'=sy$, or $s'=xsy$ for some
$x,y\in S$. If for every $s\in S$, there are only finitely many elements $\mathscr{J}$-above $s$, then $S$ is
 called finite $\mathscr{J}$-above.

In particular, the above results imply that for any $s \in \mathcal{B}(m,n)$, the elements in $\{s' \in \mathcal{B}(m,n) \mid
s' \geqslant_{\mathscr{J}} s\}$ together with zero $\square$ form a finite semigroup. Let us call this semigroup 
$S^{\mathscr{J}}_s$.

\begin{conjecture}
\label{conjecture.stable}
$(\mathsf{KR}(S^{\mathscr{J}}_s),A)$ satisfies the stability condition~\eqref{equation.stable1}. 
\end{conjecture}

Conjecture~\ref{conjecture.stable} should follow from the results in~\cite{McCammond.1991,McCammond.2001}.
Moreover, for each element $s\in \mathcal{B}(m,n)$, the language of words which 
represent $s$ is regular and can be described by a single Kleene expression without 
unions~\cite[Theorem~8.11]{McCammond.2001}.

\begin{definition}
If $S$ is a finite $\mathscr{J}$-above $A$-semigroup and $w\in A^+$, then the straight line automaton $\mathsf{str}^S(w)$
is the path $w$ together with the strong components (that is, $\mathscr{R}$-classes) of its prefixes.
\end{definition}

\begin{theorem}\cite{McCammond.1991,Guba.1993,Guba.1993a,McCammond.2001}
For $m\geqslant 3$, the set of all words equivalent to $w\in A^+$ under the relations $t^m=t^{m+n}$ is a regular language 
given by a straight line $\mathsf{str}^{\mathcal{B}(m,n)}(w)$.
\end{theorem}

Given $\mathsf{str}^{\mathcal{B}(m,n)}(w)$ for $w \in A^+$, we can construct a semigroup $S_w$ as follows. Consider 
all factors of the accepted words of $\mathsf{str}^{\mathcal{B}(m,n)}(w)$ including the empty word. These are the elements 
of $S_w$, under the equivalence relation $w_1 \equiv w_2$ if the relation $t^m = t^{m+n}$ can be used, in addition to the 
sink state $\square$. Multiplication in $\mathsf{RCay}(\mathsf{KR}(S_w),A)$ is given by $w\cdot a = wa$ for 
$a \in A$ if $wa$ is another factor of an accepted word and $\square$ otherwise. 

\begin{example}
In $S=\mathcal{B}(n,1)$ with $A=\{a,b\}$, consider $\mathsf{str}^S(w)$ for $w=(ab)^n$:
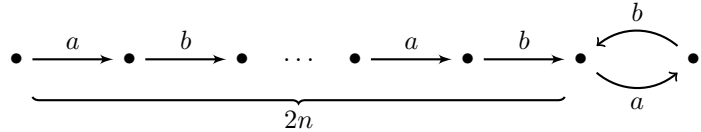
\[
\begin{tikzpicture}[auto]
\node (A) at (0, 0) {$\bullet$};
\node (B) at (1.5,0) {$\bullet$};
\node (C) at (3,0) {$\bullet$};
\node (D) at (4.5,0) {$\bullet$};
\node (E) at (6,0) {$\bullet$};
\node (F) at (7.5,0) {$\bullet$};
\node (G) at (9,0) {$\bullet$};
\node (CC) at (3.75,0) {$\hdots$};
\draw[edge,thick] (A) -- (B) node[midway, above] {$a$};
\draw[edge,thick] (B) -- (C) node[midway, above] {$b$};
\draw[edge,thick] (D) -- (E) node[midway, above] {$a$};
\draw[edge,thick] (E) -- (F) node[midway, above] {$b$};
\path (F) edge[->,thick, bend right = 40] node[midway,below] {$a$} (G)
         (G) edge[->,thick, bend right = 40] node[midway,above] {$b$} (F);
\draw [thick,decoration={brace, mirror, raise=0.5cm}, decorate] (A) -- (F)
node [pos=0.5,anchor=north,yshift=-0.55cm] {$2n$}; 
\end{tikzpicture}
\]
Then $\mathsf{RCay}(\mathsf{KR}(S_w),A)$ is given by
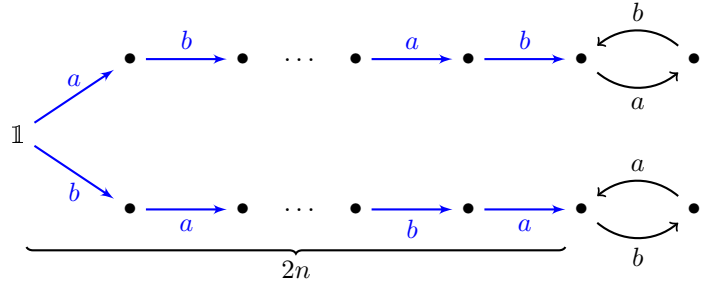
\[
\begin{tikzpicture}[auto]
\node (A) at (0, 0) {$\mathbbm{1}$};
\node (B) at (1.5,1) {$\bullet$};
\node (C) at (3,1) {$\bullet$};
\node (D) at (4.5,1) {$\bullet$};
\node (E) at (6,1) {$\bullet$};
\node (F) at (7.5,1) {$\bullet$};
\node (G) at (9,1) {$\bullet$};
\node (CC) at (3.75,1) {$\hdots$};
\draw[edge,thick,blue] (A) -- (B) node[midway, above] {$a$};
\draw[edge,thick,blue] (B) -- (C) node[midway, above] {$b$};
\draw[edge,thick,blue] (D) -- (E) node[midway, above] {$a$};
\draw[edge,thick,blue] (E) -- (F) node[midway, above] {$b$};
\path (F) edge[->,thick, bend right = 40] node[midway,below] {$a$} (G)
         (G) edge[->,thick, bend right = 40] node[midway,above] {$b$} (F);
\node (Ap) at (0,-1) {};
\node (Bp) at (1.5,-1) {$\bullet$};
\node (Cp) at (3,-1) {$\bullet$};
\node (Dp) at (4.5,-1) {$\bullet$};
\node (Ep) at (6,-1) {$\bullet$};
\node (Fp) at (7.5,-1) {$\bullet$};
\node (Gp) at (9,-1) {$\bullet$};
\node (CCp) at (3.75,-1) {$\hdots$};
\draw[edge,thick,blue] (A) -- (Bp) node[midway, below] {$b$};
\draw[edge,thick,blue] (Bp) -- (Cp) node[midway, below] {$a$};
\draw[edge,thick,blue] (Dp) -- (Ep) node[midway, below] {$b$};
\draw[edge,thick,blue] (Ep) -- (Fp) node[midway, below] {$a$};
\path (Fp) edge[->,thick, bend right = 40] node[midway,below] {$b$} (Gp)
         (Gp) edge[->,thick, bend right = 40] node[midway,above] {$a$} (Fp);
\draw [thick,decoration={brace, mirror, raise=0.5cm}, decorate] (Ap) -- (Fp)
node [pos=0.5,anchor=north,yshift=-0.55cm] {$2n$}; 
\end{tikzpicture}
\]
with arrows going to $\square$ omitted. This graph is stable under the McCammond
expansion and satisfies~\eqref{equation.stable1}. The normal forms are
\begin{equation*}
	\mathcal{N}(\mathsf{KR}(S_w),A, \{\square\}) = \{ (ab)^j b, (ab)^k aa, (ba)^j a, (ba)^k bb \mid 
	0 < j \leqslant n, \; 0\leqslant k \leqslant n\}.
\end{equation*}
We have 
\begin{equation*}
\begin{split}
	\mathsf{NF}^{-1}((ab)^j b) &= (ab)^j b \qquad \qquad \text{for $0< j<n$,}\\
	\mathsf{NF}^{-1}((ab)^ka a) &= (ab)^k a a \; \; \quad \qquad \text{for $0\leqslant k<n$,}\\
	\mathsf{NF}^{-1}((ab)^n b) &= (ab)^n (ab)^\star b,\\
	\mathsf{NF}^{-1}((ab)^n aa) &= (ab)^n (ab)^\star a a,
\end{split}
\end{equation*}
and similarly with $a$ and $b$ interchanged. Hence the stationary distribution is given by
\begin{equation*}
\begin{aligned}
\Psi_{(ab)^j b} &= x_a^j x_b^{j+1}, & \Psi_{(ab)^k aa} &= x_a^{k+2} x_b^k 
\qquad \qquad (0<j<n,\; 0\leqslant k<n)\\
\Psi_{(ab)^n b} &= \frac{x_a^n x_b^{n+1}}{1-x_a x_b}, &
\Psi_{(ab)^n a a} &= \frac{x_a^{n+2} x_b^n }{1-x_a x_b},
\end{aligned}
\end{equation*}
and similarly with $a$ and $b$ interchanged. Using 
\[
	\sum_{j=1}^{n-1} x_a^j x_b^{j+1} = x_b \frac{1-x_a^n x_b^n}{1-x_a x_b} - x_b
	\quad \text{and} \quad
	\sum_{k=0}^{n-1} x_a^{k+2} x_b = x_a^2 \frac{1-x_a^n x_b^n}{1-x_a x_b}
\]
it can be checked that the stationary distributions add to one.
\end{example}

\bibliographystyle{alpha}
\bibliography{paper}{}

\end{document}